\documentclass[11pt]{article}
	
	\usepackage{caption}
    \usepackage{url}
    \usepackage{verbatim}
	\usepackage{amsmath}
    \usepackage{color}
	\usepackage{amsthm}
	\usepackage{amsfonts}
	\usepackage{graphicx}
	\usepackage{nicefrac}

	\def\captionfont{\setb@se{11pt}\protect\footnotesize}
    \def\captionfont{\protect\footnotesize}

    \newcommand{\iprd}[2]{\left( #1 , #2 \right)}
    \newcommand{\iprdmone}[2]{\left( #1 , #2 \right)_{-1,h}}
    \newcommand{\Tauh}{\mathcal{T}_h}
    
    \newcommand{\aiprd}[2]{a\left( #1 , #2 \right)}
    \newcommand{\bform}[3]{b\left( #1 , #2, #3 \right)}
    \newcommand{\cform}[2]{c\left( #1 , #2 \right)}

    \newcommand{\Soh}{\mathring{S}_h}

	\def\norm#1#2{\left\| #1 \right\|_{#2}}

	\newcommand{\avephio}{\overline{\phi}_0}

	\newcommand{\hatphi}{\hat{\phi}}
	\newcommand{\hatmu}{\hat{\mu}}
	\newcommand{\hatu}{\hat{\bf u}}
	\newcommand{\hatp}{\hat{p}}
	\newcommand{\hatxi}{\hat{\xi}}
	\newcommand{\dtau}{\delta_\tau}
	
	\newcommand{\phih}{\phi_h}
	\newcommand{\varphih}{\varphi_h}
	\newcommand{\bu}{{\bf u}}
	\newcommand{\bv}{{\bf v}}
	\newcommand{\buh}{{\bf u}_h}
	
	\newcommand{\muh}{\mu_h}
	\newcommand{\xih}{\xi_h}
	
	\newcommand{\eAphi}{\mathcal{E}_a^{\phi}}
	\newcommand{\ehphi}{\mathcal{E}_h^{\phi}}
	\newcommand{\ephi}{\mathcal{E}^{\phi}}
	\newcommand{\eAmu}{\mathcal{E}_a^{\mu}}
	\newcommand{\ehmu}{\mathcal{E}_h^{\mu}}
	\newcommand{\emu}{\mathcal{E}^{\mu}}
	\newcommand{\eAu}{\mathcal{E}_a^{\bf u}}
	\newcommand{\ehu}{\mathcal{E}_h^{\bf u}}
	\newcommand{\eu}{\mathcal{E}^{\bf u}}
	
	\newcommand{\ehp}{\mathcal{E}_h^p}
	
	\newcommand{\eAxi}{\mathcal{E}_a^{\xi}}
	\newcommand{\ehxi}{\mathcal{E}_h^{\xi}}
	\newcommand{\sigphi}{\sigma^\phi}
	\newcommand{\sigu}{\sigma^{\bf u}}

	\newtheorem{thm}{Theorem}[section]

	\newtheorem{lem}[thm]{Lemma}
	\newtheorem{rem}[thm]{Remark}

\begin{document}
\title{Analysis of a Mixed Finite Element Method for a Cahn-Hilliard-Darcy-Stokes System}

	\author{
Amanda E. Diegel\thanks{Department of Mathematics, The University of Tennessee, Knoxville, TN 37996 (diegel@math.utk.edu)},
	\and
Xiaobing H. Feng\thanks{Department of Mathematics, The University of Tennessee, Knoxville, TN 37996 (xfeng@math.utk.edu)},
	\and
Steven M. Wise\thanks{Department of Mathematics, The University of Tennessee, Knoxville, TN 37996 (swise@math.utk.edu)}}

	\maketitle
	
	\numberwithin{equation}{section}
	
	\begin{abstract}
In this paper we devise and analyze a mixed finite element method for a modified Cahn-Hilliard equation coupled with a non-steady Darcy-Stokes flow that models phase separation and coupled fluid flow in immiscible binary fluids and diblock copolymer melts.  The time discretization is based on a convex splitting of the energy of the equation. We prove that our scheme is unconditionally energy stable with respect to a spatially discrete analogue of the continuous free energy of the system and unconditionally uniquely solvable. We prove that the phase variable is bounded in $L^\infty \left(0,T;L^\infty\right)$ and the chemical potential is bounded in $L^\infty \left(0,T;L^2\right)$, for any time and space step sizes, in two and three dimensions, and for any finite final time $T$. We subsequently prove that these variables converge with optimal rates in the appropriate energy norms in both two and three dimensions.   
	\end{abstract}
	

	\section{Introduction}

Let $\Omega\subset \mathbb{R}^d$, $d=2,3$, be an open polygonal or polyhedral domain. We consider the following modified Cahn-Hilliard-Darcy-Stokes problem with natural and no-flux/no-flow boundary conditions~\cite{choksi03,choksi11,ohta86,onuki97,zvelindovsky98a,zvelindovsky98b}: 
	\begin{subequations}
	\begin{eqnarray}
\partial_t \phi = \varepsilon\Delta \mu -\nabla\cdot\left({\bf u} \phi \right)  ,\,\,&&\text{in} \,\Omega_T ,
	\label{eq:CH-mixed-a-alt}
	\\
\mu = \varepsilon^{-1}\left(\phi^3-\phi\right)-\varepsilon \Delta \phi + \xi,\,\,&&\text{in} \,\Omega_T,   
	\label{eq:CH-mixed-b-alt}
	\\
-\Delta \xi =  \theta\left(\phi-\avephio\right),\,\,&&\text{in} \,\Omega_T,  
	\label{eq:CH-mixed-c-alt}
	\\
\omega\partial_t{\bf u} -\lambda\Delta{\bf u} +\eta{\bf u} +\nabla p = \gamma \mu \nabla\phi  ,\,\,&&\text{in} \,\Omega_T ,  
	\label{eq:CH-mixed-d-alt}
	\\
\nabla\cdot {\bf u} = 0   ,\,\,&&\text{in} \,\Omega_T ,  
	\label{eq:CH-mixed-e-alt}
	\\
\partial_n \phi =\partial_n \mu = \partial_n \xi = 0, 
\,{\bf u} = {\bf 0}, \,\,&&\text{on} \, \partial \Omega\times (0,T) ,
	\label{eq:CH-mixed-f-alt}
	\end{eqnarray}
	\end{subequations}
where $\avephio = \frac{1}{|\Omega|} \int_{\Omega} \phi_0({\bf x}) d{\bf x}$. We assume that the model parameters satisfy $\varepsilon, \gamma, \lambda >0$ and $\eta, \omega, \theta \ge 0$.  In the singular limit $\omega = 0$, we drop the initial conditions for ${\bf u}$.  We remark that it is possible to replace the right-hand-side of Eq.~\eqref{eq:CH-mixed-d-alt}, the excess forcing due to surface tension, by the term $-\gamma \phi\nabla\mu$.  The equivalence of the resulting PDE model with that above can be seen by redefining the pressure appropriately.  The latter form is what was used in the recent paper~\cite{collins13}.  The model~\eqref{eq:CH-mixed-a-alt} -- \eqref{eq:CH-mixed-f-alt} can be used to describe the flow of a very viscous block copolymer fluid~\cite{choksi03, choksi11, ohta86, onuki97, zvelindovsky98a, zvelindovsky98b}.  In Figure~\ref{fig1}, we show simulation snapshots using the equations to describe the phase separation of a block-copolymer in shear flow.  The parameters are given in the caption.

	\begin{figure}[h!]
	\begin{center}
\includegraphics[width=6in]{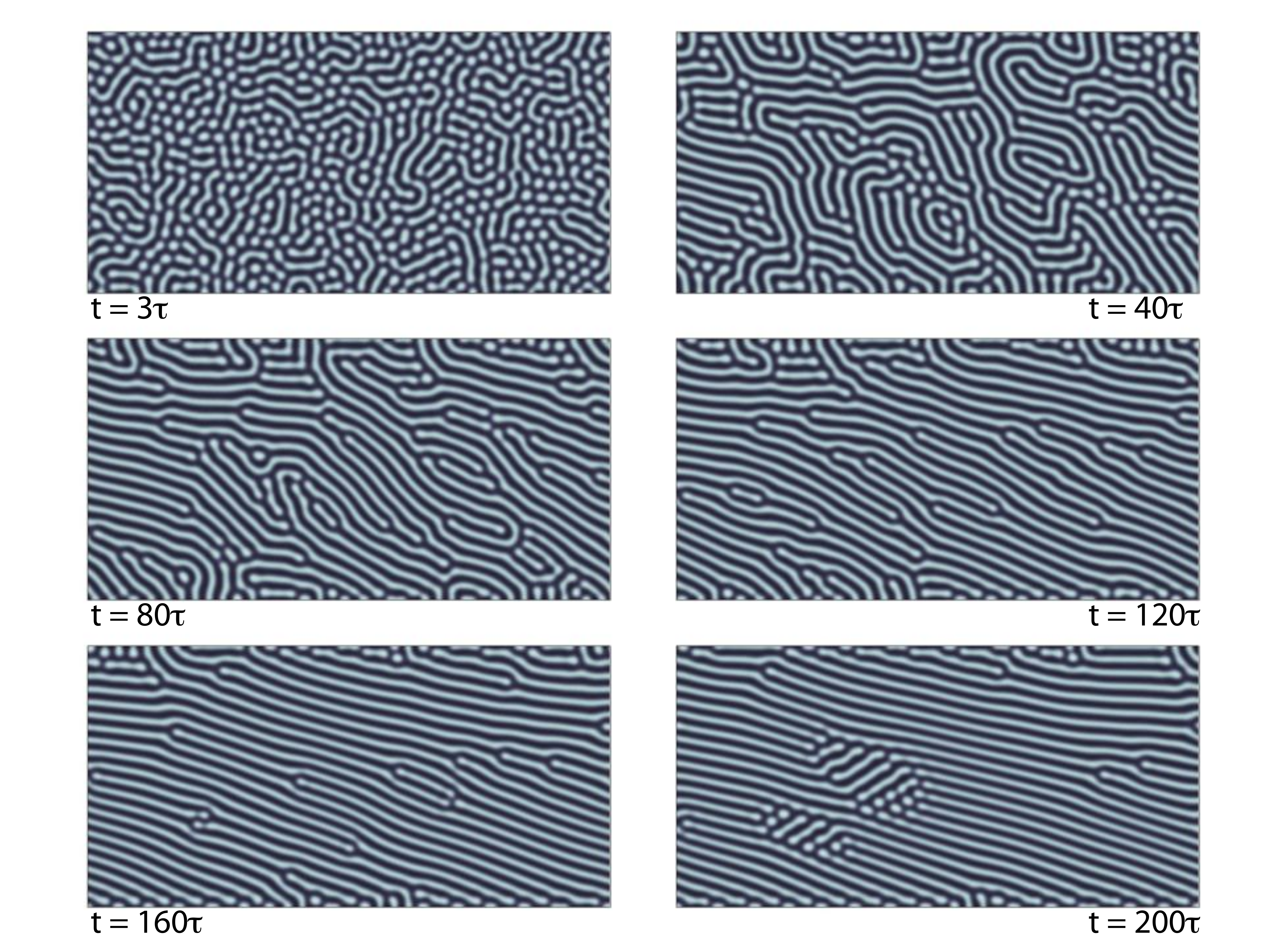}
\caption{Phase separation of a two-dimensional, very viscous block-copolymer fluid in shear flow.  The parameters are $\Omega = (0,8)\times (0,4)$, $\epsilon =0.02$, $\gamma = 0.4$, $\theta = 15000$, $\omega = 0$, $\eta = 0$, $\overline{\phi}_0 = -0.1$.  The shear velocity on the top and bottom is $\mp 2.0$, respectively.  Periodic boundary conditions are assumed in the $x$-direction.  The time unit referenced above is $\tau = 0.02$. The long-range $\theta$ term suppresses phase separation and coarsening, relative to the case $\theta=0$, and relatively long and thin phase domains emerge. This simulation compares well with other studies~\cite{zvelindovsky98a, zvelindovsky98b} that use a different dynamic density functional approach.  With a slightly larger value of $\theta$, the phase domains remain as dots and can form into hexagonal patterns, as in~\cite{zvelindovsky98b}.}
	\label{fig1}
	\end{center}
	\end{figure}

A weak formulation of \eqref{eq:CH-mixed-a-alt} -- \eqref{eq:CH-mixed-f-alt} may be written as follows: find $(\phi,\mu,\xi,{\bf u},p)$ such that
	\begin{eqnarray}
\phi &\in& L^\infty\left(0,T;H^1(\Omega)\right)\cap  L^4\left(0,T;L^\infty(\Omega)\right),
	\\
\partial_t \phi &\in&  L^2\bigl(0,T; H^{-1}(\Omega)\bigr), 
	\\
\mu &\in& L^2\bigl(0,T;H^1(\Omega)\bigr),
	\\
{\bf u} &\in& L^2\left(0,T;{\bf H}_0^1(\Omega)\right)\cap L^\infty\left(0,T;{\mathbf L}^2(\Omega)\right)
	\\
\partial_t{\bf u} &\in& L^2\left(0,T;{\bf H}^{-1}(\Omega) \right)
	\\
p &\in& L^2\left(0,T; L_0^2(\Omega)\right),
	\end{eqnarray}
and there hold for almost all $t\in (0,T)$ 
	\begin{subequations}
	\begin{align}
\langle \partial_t \phi ,\nu \rangle + \varepsilon \,\aiprd{\mu}{\nu} + \bform{\phi}{{\bf u}}{\nu} &= 0  &&\qquad \forall \,\nu \in H^1(\Omega),
	\label{weak-mch-a} 
	\\
\iprd{\mu}{\psi}-\varepsilon \,\aiprd{\phi}{\psi} - \varepsilon^{-1}\iprd{\phi^3-\phi}{\psi}  - \iprd{\xi}{\psi}&= 0  &&\qquad \forall \,\psi\in H^1(\Omega),
	\label{weak-mch-b} 
	\\
\aiprd{\xi}{\zeta}-\theta \,\iprd{\phi-\avephio}{\zeta}&= 0  &&\qquad \forall \,\zeta\in H^1(\Omega),
	\label{weak-mch-c} 
	\\
\omega \, \langle \partial_t{\bf u},{\bf v}\rangle+ \lambda\,\aiprd{{\bf u}}{{\bf v}} + \eta\iprd{ {\bf u}}{ {\bf v}} -\cform{{\bf v}}{p} - \gamma \,\bform{\phi}{{\bf v}}{\mu}&= 0  &&\qquad \forall \,{\bf v}\in {\bf H}_0^1(\Omega),
	\label{weak-mch-d} 
	\\
\cform{{\bf u}}{q} &= 0  &&\qquad \forall \,q\in L_0^2(\Omega),
	\label{weak-mch-e}
	\end{align}
	\end{subequations}
where
	\begin{equation}
\aiprd{u}{v} := \iprd{\nabla u}{\nabla v}, \quad \bform{\psi}{{\bf v}}{\nu} := \iprd{\nabla \psi \cdot {\bf v} }{\nu}, \quad \cform{{\bf v}}{q} := \iprd{\nabla \cdot{\bf v}}{q} , 
	\end{equation}
with the ``compatible" initial data
	\begin{align}
\phi(0) = \phi_0 \in H^2_N(\Omega) &:= \left\{v\in H^2(\Omega) \,\middle| \,\partial_n v = 0 \,\mbox{on} \,\partial\Omega \right\}, 
	\\
{\bf u}(0) = {\bf u}_0 \in {\bf V} &:= \left\{{\bf v}\in {\bf H}_0^1(\Omega) \,\middle| \,\cform{{\bf v}}{q} = 0 , \,\forall \,q\in L_0^2(\Omega) \right\}.
	\end{align}
Here we use the notations  $H^{-1}(\Omega) := \left(H^1(\Omega)\right)^*$, ${\bf H}_0^1(\Omega):= \left[H_0^1(\Omega)\right]^d$, ${\bf H}^{-1}(\Omega) := \left({\bf H}_0^1(\Omega)\right)^*$, and $\langle  \, \cdot \, , \, \cdot \, \rangle$ is the duality paring between $H^{-1}$ and $H^1$ in the first instance  and the duality paring between ${\bf H}^{-1}(\Omega)$ and $\left({\bf H}_0^1(\Omega)\right)^*$ in the second.  Throughout the paper, we use the notation $\chi(t) := \chi(\, \cdot \, , t)\in X$, which views a spatiotemporal function as a map from the time interval $[0,T]$ into an appropriate Banach space, $X$.  The system \eqref{weak-mch-a} -- \eqref{weak-mch-e} is mass conservative:  for almost every $t\in[0,T]$, $\iprd{\phi(t)-\phi_0}{1} = 0$.  This observation rests on the fact that $b(\phi,{\bf u},1) = 0$, for all $\phi\in L^2(\Omega)$ and all ${\bf u}\in{\bf V}$.  Observe that the homogeneous Neumann boundary conditions associated with the phase variables $\phi$, $\mu$, and $\xi$ are natural in this mixed weak formulation of the problem.  

The existence of weak solutions is a straightforward exercise using the compactness/energy method.  See, for example,~\cite{feng12}.  To define the energy for this system, we need a norm on a subspace of $H^{-1}(\Omega)$. With $L_0^2(\Omega)$ denoting those functions in $L^2(\Omega)$ with zero mean, we set
	\begin{equation}
\mathring{H}^1(\Omega) = H^1(\Omega)\cap L_0^2(\Omega) \quad \text{and} \quad \mathring{H}^{-1}(\Omega) :=\left\{v\in H^{-1}(\Omega) \,\middle| \,\langle v , 1 \rangle = 0  \right\}.
	\end{equation}
We define a linear operator $\mathsf{T} : \mathring{H}^{-1}(\Omega) \rightarrow \mathring{H}^1(\Omega)$ via the following variational problem: given $\zeta\in \mathring{H}^{-1}(\Omega)$, find $\mathsf{T}(\zeta)\in \mathring{H}^1(\Omega)$ such that
	\begin{equation}
\aiprd{\mathsf{T}(\zeta)}{\chi} = \langle \zeta, \chi\rangle \qquad \forall \,\chi\in \mathring{H}^1(\Omega).
	\end{equation}
$\mathsf{T}$ is well-defined, as is guaranteed by the Riesz Representation Theorem.  The following facts are easily established.

	\begin{lem}
	\label{lem-negative-norm}
Let $\zeta,\, \xi \in\mathring{H}^{-1}(\Omega)$ and, for such functions, set
	\begin{equation}
\left(\zeta,\xi\right)_{H^{-1}} :=\aiprd{ \mathsf{T}(\zeta)}{\mathsf{T}(\xi)} =\iprd{\zeta}{\mathsf{T}(\xi)}  = \iprd{\mathsf{T}(\zeta)}{\xi}.
	\label{crazy-inner-product}
	\end{equation}
$\left(\, \cdot\, ,\, \cdot\, \right)_{H^{-1}}$ defines an inner product on $\mathring{H}^{-1}(\Omega)$, and the induced norm is equal to the operator norm:
	\begin{equation}
\norm{\zeta}{H^{-1}} := \sqrt{\left(\zeta,\zeta\right)_{H^{-1}}} = \sup_{0\ne \chi\in\mathring{H}^1} \frac{\langle \zeta , \chi\rangle}{\norm{\nabla\chi}{L^2}} . 
	\label{crazy-norm-minus-one}
	\end{equation}
Consequently, for all $\chi\in H^1(\Omega)$ and all $\zeta\in\mathring{H}^{-1}(\Omega)$, 
	\begin{equation}
\left|\langle \zeta , \chi\rangle\right| \le \norm{\zeta}{H^{-1}} \norm{\nabla\chi}{L^2} .
	\end{equation}
Furthermore, for all $\zeta\in L_0^2(\Omega)$, we have the Poincar\'{e} type inequality
	\begin{equation}
\norm{\zeta}{H^{-1}} \le C  \norm{\zeta}{L^2},
	\end{equation}
where $C>0$ is the usual Poincar\'{e} constant.
	\end{lem}

Consider  the energy
	\begin{eqnarray}
E({\bf u},\phi) &=& \frac{\omega}{2\gamma} \norm{{\bf u}}{L^2}^2+ \frac{1}{4\varepsilon}\norm{\phi^2-1}{L^2}^2 +\frac{\varepsilon}{2} \norm{\nabla \phi}{L^2}^2 +\frac{\theta}{2}\norm{\phi-\avephio}{H^{-1}}^2
	\nonumber
	\\
&=& \frac{\omega}{2\gamma} \norm{{\bf u}}{L^2}^2+ \frac{1}{4\varepsilon}\norm{\phi}{L^4}^4 - \frac{1}{2\varepsilon}\norm{\phi}{L^2}^2 +  \frac{|\Omega|}{4\varepsilon} +\frac{\varepsilon}{2} \norm{\nabla \phi}{L^2}^2 +\frac{\theta}{2}\norm{\phi-\avephio}{H^{-1}}^2 ,
	\label{continuous-energy}
	\end{eqnarray}
which is defined for all ${\bf u} \in {\bf L}^2(\Omega)$ and $\phi\in\mathcal{A}:= \left\{\phi\in H^1(\Omega)\,\middle| \,\iprd{\phi-\avephio}{1}=0\right\}$. Clearly, if $\theta \ge 0$, then $E({\bf u},\phi) \ge 0$ for all ${\bf u} \in {\bf L}^2(\Omega)$ and $\phi\in \mathcal{A}$.  For arbitrary $\theta\in\mathbb{R}$, $\varepsilon >0$, ${\bf u} \in {\bf L}^2(\Omega)$, and $\phi\in\mathcal{A}$, there exist positive constants $M_1=M_1(\varepsilon,\theta)$ and $M_2=M_2(\varepsilon,\theta)$ such that
	\begin{equation} 
0 < M_1\left(\norm{{\bf u}}{L^2}^2+ \norm{\phi}{H^1}^2 \right) \le E({\bf u},\phi) +M_2.
	\end{equation}
It is straightforward to show that weak solutions of \eqref{weak-mch-a} -- \eqref{weak-mch-e} dissipate the energy \eqref{continuous-energy}. In other words, \eqref{eq:CH-mixed-a-alt} -- \eqref{eq:CH-mixed-f-alt} is a conserved gradient flow with respect to the energy \eqref{continuous-energy}. Precisely, for any $t\in[0,T]$, we have the energy law
	\begin{equation}
E({\bf u}(t),\phi(t)) +\int_0^t\left(\frac{\lambda}{\gamma}\norm{\nabla{\bf u}(s)}{L^2}^2 +  \frac{\eta}{\gamma}\norm{{\bf u}(s)}{L^2}^2+\varepsilon\norm{\nabla\mu(s)}{L^2}^2\right) ds = E({\bf u}_0,\phi_0) .
	\label{pde-energy-law}
	\end{equation}
Formally, one can also easily demonstrate that $\mu$ in \eqref{eq:CH-mixed-b-alt} is the variational derivative of $E$ with respect to $\phi$. In symbols, $\mu = \delta_\phi E$.  

The energy typically ``prefers" the fluid phase states $\phi \approx \pm 1$ (the pure phases) separated by a diffuse interface of (small) thickness $\epsilon$.  However, the long-range energy described by the last term can change this picture~\cite{aristotelous13a, choksi11, choksi03, ohta86}.  Specifically, when $\theta > 0$, the energy term $\frac{\theta}{2}\norm{\phi-\avephio}{H^{-1}}^2$ in \eqref{continuous-energy} is convex and stabilizing, and this energy tends to stabilize (or suppress) both the phase separation and the coarsening processes. This is observed in Fig.~\ref{fig1}. If $\theta < 0$ the term is concave and destabilizing.  In this case, the process of phase separation will be enhanced.  Herein we assume that $\theta \ge 0$.  
	
In the case that ${\bf u}\equiv 0$ -- which occurs if $\gamma=0$ -- the model \eqref{eq:CH-mixed-a-alt} -- \eqref{eq:CH-mixed-f-alt} reduces to the modified Cahn-Hillard equation~\cite{choksi03,choksi11} which was analyzed by  Aristotelous~\emph{et al.}~\cite{aristotelous13a}. Their scheme was comprised of a convex splitting method for time discretization and a discontinuous galerkin finite element method for space discretization.  They showed that their mixed, fully discrete scheme was unconditionally energy stable, unconditionally uniquely solvable, and optimally convergent in the energy norm in two-dimensions.  Collins~\emph{et al.}~\cite{collins13} used a convex-splitting method in time and a finite difference method in space to devise an energy stable method for a system similar to \eqref{eq:CH-mixed-a-alt} -- \eqref{eq:CH-mixed-f-alt}, though they did not prove convergence or error stimulates. 

In the present work, we wish to generalize both~\cite{aristotelous13a, collins13}, in a certain sense, and prove convergence and optimal error estimates for a convex-splitting scheme for \eqref{eq:CH-mixed-a-alt} -- \eqref{eq:CH-mixed-f-alt} using a standard Galerkin finite element discretization of space, as opposed to the discontinuous Galerkin finite element setting used in~\cite{aristotelous13a} or the finite difference setting used in~\cite{collins13}.   We remark that, as in~\cite{aristotelous13a}, the $\theta$ ``growth" term makes the analysis somewhat different from other works involving the Cahn-Hilliard equations. Specifically, analysis of this term requires the introduction of a discrete negative norm as in~\cite{aristotelous13a}. 

The goal of this paper is to construct a fully discrete, unconditionally energy stable, unconditionally uniquely solvable, and convergent mixed finite element scheme for the Cahn-Hilliard-Darcy-Stokes system \eqref{eq:CH-mixed-a-alt} -- \eqref{eq:CH-mixed-f-alt}.  Energy stability means that the discrete solutions dissipate the energy in time, in analogy with the PDE model.  The convex-splitting approach that we will employ was popularized by Eyre~\cite{eyre98} and has been used by many others \cite{collins13, elliott89b, feng12, wang11, wise09a}. In the convex-splitting framework, one treats the contribution from the convex part implicitly and the contribution from the concave part explicitly. This treatment promotes the scheme's energy stability and unique solvability, both properties being unconditional with respect to the time and space step sizes.

Galerkin numerical methods for the Cahn-Hilliard-Navier-Stokes and the Allen-Cahn-Navier-Stokes have been investigated in the recent papers~\cite{abels12, feng06, feng07b, feng12, kay08, kay07, grun13, grun14, shen10a, shen10b}.  The rigorous analyses of numerical schemes -- mostly for the matched-density \emph{CHNS} system -- can be found in~\cite{feng06, feng07b, feng12, kay08, kay07, grun13, shen10a, shen10b}.  Specifically, there have been convergence proofs for these schemes, but all of these analyses focus on two types of limited convergence results: (i) error estimates and convergence rates for the semi-discrete setting (time continuous)~\cite{feng07b, kay08} and/or (ii) abstract convergence results with no convergence rates~\cite{feng07b, feng12, grun13, kay08}.  \emph{Optimal error estimates for the fully discrete schemes are lacking in the literature.}

Kay~\emph{et al.} develop both a semi-discrete and a fully discrete mixed finite element method for the Cahn-Hilliard-Navier-Stokes system of equations. For the semi-discrete model, they were able to show unconditional stabilities resulting from the discrete energy law.  For the fully discrete model, they use a first order implicit-explicit Euler method to discretize time and were able to show conditional energy stability, with a restriction on the time step.  They were able to obtain optimal error (convergence) rates for the semi-discrete model, but only an abstract convergence for the fully discrete model. In \cite{grun13}, Gr\"{u}n proves the abstract convergence of a fully discrete finite element scheme for a diffuse interface model for two-phase flow of incompressible, viscous fluids with different mass densities.  No convergence rates were presented in his paper. Feng~\cite{feng06} presented a fully discrete mixed finite element method for the Cahn-Hilliard-Navier-Stokes system of equations. The time discretization used is a first order implicit Euler with the exception of a stabilization term which is treated explicitly.  Conditional stability for the basic energy law is developed along with abstract convergence of the finite element model to the PDE model. However, no additional stability estimates are presented beyond the estimates achieved from the energy law. Feng~\emph{et al.}~\cite{feng07b} develop both a semi-discrete and fully discrete finite element method model for the Non-steady-Stokes-Allen-Cahn system of equations. For both the semi-discrete and fully discrete models, conditional energy stability is developed. Optimal error estimates are obtained for the semi-discrete scheme (time-continuous) while abstract convergence is proven for the fully discrete model.  Shen and Yang~\cite{shen10b} derived a phase field model for two-phase incompressible flows with variable density without assuming a small density ratio and constructed several efficient time discretization schemes.  For each of the schemes, a discrete energy law similar to the continuous energy law is established.  Finally, they present several numerical results to illustrate the effectiveness and correctness of their numerical schemes. However, no convergence or error analyses are presented in the paper.  

The work presented in this paper on the modified Cahn-Hilliard-Darcy-Stokes system is unique in the following sense. We are able to prove unconditional unique solvability, unconditional energy stability, and optimal error estimates for a fully discrete finite element scheme in three dimensions. Specifically, the stability and solvability statements we prove are \emph{completely unconditional with respect to the time and space step sizes}. In fact, all of our \emph{a priori} stability estimates hold completely independently of the time and space step sizes.  We use a bootstrapping technique to leverage the energy stabilities  to achieve unconditional $L^{\infty}(0,T; L^{\infty}(\Omega))$ stability for the phase field variable $\phi_h$ and unconditional $L^{\infty}(0,T; L^2(\Omega))$ stability for the chemical potential $\mu_h$.  With these stabilities in hand we are able to prove optimal error estimates for $\phi_h$ and $\mu_h$ in the appropriate energy norms.

The paper is organized as follows. In Section~\ref{sec:defn-and-properties}, we define our mixed finite element convex-splitting scheme and prove the unconditional solvability and stability of the scheme. In Section~\ref{sec-error-estimates}, we prove error estimates for the scheme under suitable regularity assumptions for the PDE solution. In Section~\ref{sec:numerincal-experiments}, we present the results of numerical tests that confirm the rates of convergence predicted by the error estimates.

	\section{A Mixed Finite Element Convex Splitting Scheme}
	\label{sec:defn-and-properties}
	
	\subsection{Definition of the Scheme}
	\label{subsec-defn}

Let $M$ be a positive integer and $0=t_0 < t_1 < \cdots < t_M = T$ be a uniform partition of $[0,T]$, with $\tau = t_i-t_{i-1}$, $i=1,\ldots ,M$.  Suppose ${\mathcal T}_h = \left\{ K \right\}$ is a conforming, shape-regular, quasi-uniform family of triangulations of $\Omega$.  For $r\in\mathbb{Z}^+$, define $\mathcal{M}_r^h := \left\{v\in C^0(\Omega) \, \middle| \,v|_K \in {\mathcal P}_r(K), \,\forall \,\,  K\in \mathcal{T}_h \right\}\subset H^1(\Omega)$ and $\mathcal{M}_{r,0}^h :=\mathcal{M}_r^h\cap H_0^1(\Omega)$.

For a given positive integer $q$, we set $S_h := \mathcal{M}_q^h$;  $\Soh := S_h\cap L_0^2(\Omega)$; 
	\begin{displaymath}
{\bf X}_h := \left\{ {\bf v}\in \left[C^0(\Omega)\right]^d \,\middle| \,v_i \in \mathcal{M}_{q+1,0}^h, \,i = 1, \ldots , d \right\}; \
{\bf V}_h := \left\{{\bf v} \in {\bf X}_h \,\middle| \,\left(\nabla\cdot {\bf v}, w \right) = 0, \forall \,\,  w\in \Soh  \right\}.
	\end{displaymath}
Note that ${\bf V}_h \not\subset {\bf V}$, in general. Our mixed convex-splitting scheme is defined as follows:  for any $1\le m\le M$, given  $\phih^{m-1} \in S_h$, $\buh^{m-1}\in{\bf X}_h$, find $\phih^m,\muh^m\in S_h$,  $\xi_h^m, p_h^m \in \Soh$, and $\buh^m\in{\bf X}_h$,  such that  
	\begin{subequations}
	\begin{align}
\iprd{\dtau \phih^m}{\nu} + \varepsilon \,\aiprd{\muh^m}{\nu} + \bform{\phih^{m-1}}{\buh^m}{\nu}= \,& 0 , \, & \forall \, \nu \in S_h ,
	\label{scheme-a}
	\\
\varepsilon^{-1} \,\iprd{\left(\phih^m\right)^3 -\phih^{m-1}}{\psi} + \varepsilon \,\aiprd{\phih^m}{\psi} -\iprd{\muh^m}{\psi} +\iprd{\xi_h^m}{\psi} = \,& 0 , \,&\forall \, \psi\in S_h ,
	\label{scheme-b}
	\\
\aiprd{\xi_h^m}{\zeta}-\theta \,\iprd{\phih^m-\avephio}{\zeta}  = \,& 0 , \,&\forall \, \zeta\in S_h ,
	\label{scheme-c} 
	\\
\iprd{\delta_\tau{\bf u}^m_h}{{\bf v}}+ \lambda \,\aiprd{{\bf u}^m_h}{{\bf v}} + \eta \,\iprd{ \buh^m}{ {\bf v}} -\cform{{\bf v}}{p^m_h} - \gamma \,\bform{\phih^{m-1}}{{\bf v}}{\muh^m} = \,& 0 ,\, &  \forall \, {\bf v}\in {\bf X}_h,
	\label{scheme-d} 
	\\
\cform{\buh^m}{q} = \,& 0 ,\, & \forall \, q\in \Soh,
	\label{scheme-e} 
	\end{align}
	\end{subequations}
where 
	\begin{equation}
\dtau \phih^m := \frac{\phih^m-\phih^{m-1}}{\tau}, \quad  \phih^0 := R_h \phi_0 , \quad \buh^0 := {\bf P}_h\bu_0.
	\end{equation}
The operator $R_h: H^1(\Omega) \to S_h$ is the Ritz projection:
	\begin{equation}
\aiprd{R_h\phi - \phi}{\chi} = 0, \quad \forall \, \chi\in S_h, \quad \iprd{R_h \phi-\phi}{1}=0.
	\end{equation}
The operator $\left({\bf P}_h,P_h\right): {\bf V}\times L_0^2 \to {\bf V}_h\times \Soh$ is the Darcy-Stokes projection:
	\begin{align}
\lambda \,\aiprd{{\bf P}_h \bu- \bu}{{\bf v}} + \eta \,\iprd{{\bf P}_h \bu- \bu} {{\bf v}} - \cform{{\bf v}}{P_h p -p} &= 0, \quad \forall \, \bv\in {\bf X}_h, 
	\\
\cform{{\bf P}_h \bu- \bu}{q} &= 0, \quad \forall \, q\in \Soh.
	\end{align}

	\begin{rem}
To shorten the presentation, we have set $\omega = 1$ (appearing in \eqref{eq:CH-mixed-d-alt}). We remark that the more general case $\omega >0$ can be considered in the analysis without any significant changes.  Additionally, with some slight modifications, the singular limit case, $\omega = 0$, can be covered in the analysis that follows. In this case, one looses the stability ${\bf u}_h\in L^\infty(0,T;L^2(\Omega))$.  For perspective, the analysis of Feng~\emph{et al.}~\cite{feng07b} requires ${\bf u}_h\in L^\infty(0,T;L^2(\Omega))$. 
	\end{rem}

	\begin{rem}
Note that $\iprd{\phih^0-\avephio}{1} =0$, where $\avephio$ is the initial mass average, which in the typical case, satisfies $|\avephio| \le 1$. We also point out that, appealing to  \eqref{scheme-a} and \eqref{scheme-e}, we have $\iprd{\phih^m-\avephio}{1} =0$, for all $m = 1, \ldots , M$, which follows because $\aiprd{\mu}{1} = 0$, for all $\mu\in S_h$, and $\bform{\phi_h}{{\bf u}}{1} = 0$, for all $\phi\in S_h$ and all ${\bf u}\in {\bf V}_h$.
	\end{rem}
	
	\begin{rem}
	\label{rem:initial-projection}
The elliptic projections are used in the initialization for simplicity in the forthcoming analysis. We can use other (simpler) projections in the initialization step, as long as they have good approximation properties.
	\end{rem}

	\begin{rem}
Note that it is not necessary for solvability and some basic energy stabilities that the $\mu$--space and the $\phi$--space be equal.  However, the proofs of the higher-order stability estimates, in particular those in Lemma~\ref{lem-a-priori-stability}, do require the equivalence of these spaces.  Mass conservation of the scheme requires some compatibility of the $p$--space with that of the $\phi$--space, to obtain $\bform{\phi_h}{{\bf u}}{1} = 0$. For the flow problem we have chosen the inf-sup-stable Taylor-Hood element.  One can also use the simpler MINI element.  Recall that the stability of the Taylor-Hood element typically requires that the family of meshes ${\mathcal T}_h$ has the property that no tetrahedron/triangle in the mesh has more than one face/edge on the boundary~\cite{brenner08}.
	\end{rem}

Now, we can define a scheme that is equivalent to that above.  For any $1\le m\le M$, given  $\varphih^{m-1} \in S_h$, $\buh^{m-1}\in{\bf X}_h$, find $\varphih^m,\muh^m\in S_h$,  $\xi_h^m \in \Soh$, $\buh^{m,0},\buh^{m,1}\in{\bf X}_h$, $p_h^{m,0},p_h^{m,1}\in \Soh$,  such that 
	\begin{subequations}
	\begin{align}
\lambda \,\aiprd{{\bf u}^{m,0}_h}{{\bf v}} +\left(\eta+\frac{1}{\tau}\right) \iprd{ \buh^{m,0}}{ {\bf v}} -\cform{{\bf v}}{p^{m,0}_h} -\frac{1}{\tau} \,\iprd{{\bf u}^{m-1}_h}{{\bf v}}  = \, & 0 ,\, &  \forall \, {\bf v}\in {\bf X}_h,
	\label{scheme-d-mean-zero-u0} 
	\\
\cform{\buh^{m,0}}{q} = \,& 0 ,\, & \forall \, q\in \Soh,
	\label{scheme-e-mean-zero-u0} 
	\end{align}
	\end{subequations}
and
	\begin{subequations}
	\begin{align}
\iprd{\frac{\varphih^m - \varphi_{h,\star}^{m-1}}{\tau}}{\nu} + \varepsilon \,\aiprd{\muh^m}{\nu} + \bform{\varphih^{m-1}}{{\bf u}_h^{m,1}}{\nu}  = \,& 0 ,   & \forall \, \nu \in S_h ,
	\label{scheme-a-mean-zero}
	\\
\varepsilon^{-1} \,\iprd{\left(\varphih^m+\avephio \right)^3 -\varphih^{m-1}-\avephio}{\psi} + \varepsilon \,\aiprd{\varphih^m}{\psi} -\iprd{\muh^m}{\psi} +\iprd{\xi_h^m}{\psi} = \,& 0 ,  &\forall \, \psi\in S_h ,
	\label{scheme-b-mean-zero}
	\\
\aiprd{\xi_h^m}{\zeta} - \theta \,\iprd{\varphih^m}{\zeta}  = \,& 0 ,  &\forall \, \zeta\in S_h ,
	\label{scheme-c-mean-zero} 
	\\
\lambda \,\aiprd{{\bf u}^{m,1}_h}{{\bf v}} + \left( \eta +\frac{1}{\tau}\right) \iprd{{\bf u}^{m,1}_h}{ {\bf v}} - \cform{{\bf v}}{p^{m,1}_h} - \gamma \,\bform{\varphih^{m-1}}{{\bf v}}{\muh^m} = \,& 0 ,  &  \forall \, {\bf v}\in {\bf X}_h,
	\label{scheme-d-mean-zero} 
	\\
\cform{\buh^{m,1}}{q} = \,& 0 ,  & \forall \, q \in \Soh,
	\label{scheme-e-mean-zero} 
	\end{align}
	\end{subequations}
where
	\begin{equation}
\varphi_{h,\star}^{m-1} := \varphih^{m-1} - \tau\mathcal{Q}_h\left( \nabla \varphih^{m-1} \cdot {\bf u}_h^{m,0}\right) \in S_h,
	\label{phi-m-1-star}
 	\end{equation}
and $\mathcal{Q}_h: L^2(\Omega) \to S_h$ is the $L^2$ projection, \emph{i.e.}, $\iprd{\mathcal{Q}_h \nu-\nu}{\chi}=0$, for all $\chi\in S_h$. For the initial data, we set
	\begin{equation}
\varphih^0 := R_h \phi_0 - \avephio  , \quad \buh^0 := {\bf P}_h\bu_0.
	\end{equation}

Hence, $\iprd{\varphih^0}{1}=0$. By setting $\nu\equiv 1$ in \eqref{scheme-a} and \eqref{scheme-a-mean-zero} and observing that $\aiprd{\varphi}{1} = 0$ for all $\varphi\in S_h$, one finds that, provided solutions for the two schemes exist, they are related via 
	\begin{equation}
\varphih^m +\avephio = \phih^m, \quad \varphih^m\in\Soh,  \quad {\bf u}_h^m = {\bf u}_h^{m,0}+ {\bf u}_h^{m,1} \in {\bf X}_h, \quad  p^m_h = p^{m,0}_h + p^{m,1}_h \in \Soh
	\end{equation}
for all $1\le m\le M$.  The variables $\muh^m$ and $\xi_h^m$ are the same as before.  Note that the average mass of $\muh^m$ will change with the time step $m$, \emph{i.e.}, $\iprd{\muh^m}{1} \ne \iprd{\muh^{m-1}}{1}$, in general.

	\begin{rem}
The utility of this new, equivalent formulation is that we can straightforwardly show its unconditional unique solvability by convex optimization methods.  Our arguments require that the velocity ${\bf u}_h^{m,1}$ is a linear function of $\mu_h^m$, as is the case in \eqref{scheme-d-mean-zero}.  (See Lemma~\ref{lem-bilinear-ell}.)  This was not the case in \eqref{scheme-d}, where ${\bf u}_h^m$ is an affine function of $\mu_h^m$.
	\end{rem}
	
	\subsection{Unconditional Solvability}
	\label{subsec-solvability}
In this subsection, we show that our schemes are unconditionally uniquely solvable.  We begin by defining some machinery for the solvability, as well as the stability and convergence analyses discussed later.  First, consider the invertible linear operator $\mathsf{T}_h : \Soh \rightarrow \Soh$ defined via the following variational problem: given $\zeta\in \Soh$, find $\mathsf{T}_h(\zeta)\in \Soh$ such that
	\begin{equation}
\aiprd{\mathsf{T}_h(\zeta)}{\chi} = \iprd{\zeta}{\chi} \qquad \forall \, \chi\in \Soh.
	\end{equation}
This clearly has a unique solution because $\aiprd{\, \cdot \, }{ \, \cdot \, }$ is an inner product on $\Soh$. We now wish to define a mesh-dependent ``$-1$" norm, \emph{i.e.}, a discrete analogue to the $H^{-1}$ norm.  We omit the details of the next result, the discrete analog of Lemma~\ref{lem-negative-norm}, for brevity.

	\begin{lem}
	\label{lem-negative-norm-discrete}
Let $\zeta,\, \xi \in \Soh$ and set
	\begin{equation}
\left(\zeta,\xi\right)_{-1,h} :=\aiprd{\mathsf{T}_h(\zeta)}{\mathsf{T}_h(\xi)} =\iprd{\zeta}{\mathsf{T}_h(\xi)}  = \iprd{\mathsf{T}_h(\zeta)}{\xi}.
	\label{crazy-inner-product-h}
	\end{equation}
$\left(\, \cdot\, ,\, \cdot\, \right)_{-1,h}$ defines an inner product on $\Soh$, and the induced negative norm satisfies
	\begin{equation}
\norm{\zeta}{-1,h} := \sqrt{\left(\zeta,\zeta\right)_{-1,h}} = \sup_{0\ne \chi\in\Soh} \frac{\iprd{\zeta}{\chi}}{\norm{\nabla\chi}{L^2}} . 
	\label{crazy-norm-h}
	\end{equation}
Consequently, for all $\chi\in S_h$ and all $\zeta\in\Soh$, 
	\begin{equation}
\left|\iprd{\zeta}{\chi}\right| \le \norm{\zeta}{-1,h} \norm{\nabla\chi}{L^2} .
	\label{plus-1-minus-1-estimate}
	\end{equation}
The following Poincar\'{e}-type estimate holds:
	\begin{equation}
\norm{\zeta}{-1,h} \le C  \norm{\zeta}{L^2}, \quad  \forall \, \,\zeta\in\Soh, 
	\end{equation}
for some $C >0$ that is independent of $h$.  Finally, if $\Tauh$ is globally quasi-uniform, then the following inverse estimate holds:
	\begin{equation}
\norm{\zeta}{L^2} \le C h^{-1} \norm{\zeta}{-1,h}, \quad  \forall \, \,\zeta\in\Soh,
	\end{equation}
for some $C >0$ that is independent of $h$.
	\end{lem}

	\begin{lem}
	\label{lem-bilinear-ell}
Given $\varphih^{m-1}\in\Soh$ define the bilinear form $\ell_h^m:\Soh\times\Soh\to \mathbb{R}$ via
	\begin{equation}
	\label{weak-form-L-h}
\ell_h^m(\mu,\nu):= \varepsilon \,\aiprd{\mu}{\nu} + \bform{\varphih^{m-1}}{{\bf u}}{\nu},
	\end{equation}
where, for each fixed $\mu\in\Soh$,  ${\bf u} = {\bf u}(\mu) \in {\bf X}_h$ and $p = p(\mu)\in \Soh$ solve
	\begin{subequations}
	\begin{align}
\lambda \,\aiprd{{\bf u}}{{\bf v}} + \left(\eta+\frac{1}{\tau}\right)\iprd{{\bf u}}{ {\bf v}} -\cform{{\bf v}}{p} - \gamma \,\bform{\varphih^{m-1}}{{\bf v}}{\mu} = \,& 0 ,\, &  \forall \, {\bf v}\in {\bf X}_h,
	\label{symm-u}
	\\
\cform{{\bf u}}{q} = \,& 0 ,\, & \forall \, q\in \Soh.
	\label{symm-p}
	\end{align}
	\end{subequations}
Then $\ell_h^m(\, \cdot \, , \, \cdot \, )$ is a coercive, symmetric bilinear form, and therefore, an inner product on $\Soh$.  
	\end{lem}
	\begin{proof}
The solvability and stability  of the flow problem follows from the fact that $\left({\bf X}_h, \Soh\right)$ form a stable pair for the Darcy-Stokes problem. Now, let $\mu_i\in\Soh$, $i = 1,2$.  Set ${\bf u}_i = {\bf u}(\mu_i) \in {\bf X}_h$ and $p_i = p(\mu_i)\in \Soh$, $i=1,2$, with ${\bf u}$ and $p$ defined in \eqref{symm-u} and \eqref{symm-p} above. Then with $\alpha,\beta \in \left\{1,2\right\}$, 
	\begin{subequations}
	\begin{align}
\lambda \,\aiprd{{\bf u}_\alpha}{{\bf u}_\beta} + \left(\eta+\frac{1}{\tau}\right)\iprd{{\bf u}_\alpha}{ {\bf u}_\beta} -\cform{{\bf u}_\beta}{p_\alpha} - \gamma \,\bform{\varphih^{m-1}}{{\bf u}_\beta}{\mu_\alpha} = \,& 0 ,
	\\
\cform{{\bf u}_\beta}{p_\alpha} = \,& 0 ,
	\end{align}
	\end{subequations}
and setting $\alpha = 2$, $\beta = 1$ in the last two equations, we have
	\begin{align}
\ell_h^m(\mu_1,\mu_2) &= \varepsilon \,\aiprd{\mu_1}{\mu_2} + \bform{\varphih^{m-1}}{{\bf u}_1}{\mu_2}
	\nonumber
	\\
&= \varepsilon \,\aiprd{\mu_1}{\mu_2} +  \frac{\lambda}{\gamma} \,\aiprd{{\bf u}_2}{{\bf u}_1} + \frac{\eta+\frac{1}{\tau}}{\gamma}\iprd{{\bf u}_2}{ {\bf u}_1}.
	\end{align}
It is now clear that $\ell_h^m(\, \cdot \, , \, \cdot \, )$ is a coercive, symmetric bilinear form on $\Soh$.
	\end{proof}

Owing to the last result, we can define an invertible linear operator ${\mathcal L}_{h,m} : \Soh \rightarrow  \Soh$ via the following problem: given $\zeta\in \Soh$,  find $\mu\in \Soh$ such that
	\begin{equation}
\ell_h^m\left(\mu,\nu\right) = -\left(\zeta,\nu\right)_{L^2}   \qquad \forall \, \nu\in \Soh.
	\end{equation}
This clearly has a unique solution because $\ell_h^m(\, \cdot \, , \, \cdot \, )$ is an inner product on $\Soh$. We write ${\mathcal L}_{h,m}(\mu) = -\zeta$, or, equivalently, $\mu = -{\mathcal L}_{h,m}^{-1}(\zeta)$.

We now wish to define another discrete negative norm.  Again we omit the details for the sake of brevity.

	\begin{lem}
	\label{lem-bilinear-negative-norm}
Let $\zeta,\, \xi \in \Soh$ and suppose  $\mu_\zeta,\, \mu_\xi\in\Soh$ are the unique weak  solutions to ${\mathcal L}_{h,m}\left(\mu_\zeta\right) = -\zeta$ and ${\mathcal L}_{h,m}\left(\mu_\xi\right) = -\xi$.  Define
	\begin{equation}
\left(\zeta,\xi\right)_{{\mathcal L}_{h,m}^{-1}} :=\ell_h^m\left(\mu_\zeta,\mu_\xi\right) =-\left(\zeta,\mu_\xi\right)_{L^2}  =-\left(\mu_\zeta,\xi\right)_{L^2}.
	\label{crazy-inner-product-h-ell}
	\end{equation}
$\left(\, \cdot\, ,\, \cdot\, \right)_{{\mathcal L}_{h,m}^{-1}}$ defines an  inner product on $\Soh$. The induced norm is 
	\begin{equation}
\norm{\zeta}{\mathcal{L}_{h,m}^{-1}} 
= \sqrt{\left(\zeta,\zeta\right)_{{\mathcal L}_{h,m}^{-1}}}, \quad \forall \, \zeta\in \Soh.
	\label{crazy-norm-h-ell}
	\end{equation}
	\end{lem}

Using our discrete negative norm we can define a variational problem closely related to our fully discrete scheme. To keep the discussion brief, we omit the proof of the following result that guarantees the unique solvability of \eqref{scheme-a-mean-zero} -- \eqref{scheme-e-mean-zero}.  See, for example,~\cite{feng12} for the details of a similar argument.

	\begin{lem}
	\label{lem-minimizer}
Let $\varphih^{m-1}\in\Soh$ be given.  Take $\varphi_{h,\star}^{m-1}$ as in \eqref{phi-m-1-star}.  For all $\varphih\in\Soh$, define the nonlinear functional 
	\begin{eqnarray}
G_h(\varphih) &:=& \frac{\tau}{2}\norm{\frac{\varphih-\varphi_{h,\star}^{m-1}}{\tau}}{\mathcal{L}_{h,m}^{-1}}^2 +\frac{1}{4\varepsilon}\norm{\varphih+\avephio}{L^4}^4  +\frac{\varepsilon}{2}\norm{\nabla\varphih}{L^2}^2
	\nonumber
	\\
&& -\frac{1}{\varepsilon}\iprd{\varphih^{m-1}+\avephio}{\varphih} + \frac{\theta}{2}\norm{\varphih}{-1,h}^2.
	\end{eqnarray}
$G_h$ is strictly convex and coercive on the linear subspace $\Soh$.  Consequently, $G_h$ has a unique minimizer, call it  $\varphih^m\in\Soh$.  Moreover, $\varphih^m\in\Soh$ is the unique minimizer of $G_h$ if and only if it is the unique solution to
	\begin{equation}
\varepsilon^{-1}\iprd{\left(\varphih^m+\avephio\right)^3}{\psi} +\varepsilon \,\aiprd{\varphih^m}{\psi}  - \iprd{\mu_{h,\star}^m}{\psi} +\iprd{\xi_h^m}{\psi}  = \varepsilon^{-1}\iprd{\varphih^{m-1}+\avephio}{\psi}
	\label{nonlinear-1}
	\end{equation}
for all $\psi\in\Soh$, where $\mu_{h,\star}^m,\xi_h^m \in \Soh$  are the unique solutions to
	\begin{align}
\ell_h^m\left(\mu_{h,\star}^m,\nu\right)  = -\iprd{\frac{\varphih^m-\varphi_{h,\star}^{m-1}}{\tau}}{\nu} &  \qquad \forall \,  \nu\in\Soh,
	\label{nonlinear-2}
	\\
\aiprd{\xi_h^m}{\zeta} = \theta \,\iprd{\varphih^m}{\zeta} &  \qquad \forall \,  \zeta\in\Soh.
	\label{nonlinear-3}
	\end{align}
	\end{lem}

Finally, we are in the position to prove the unconditional unique solvability of our scheme.

	\begin{thm}
	\label{thm-existence-uniqueness}
The scheme \eqref{scheme-a} -- \eqref{scheme-e}, or, equivalently, the scheme \eqref{scheme-a-mean-zero} -- \eqref{scheme-e-mean-zero}, is uniquely solvable for any  mesh parameters $\tau$ and $h$ and for any of the model parameters.
	\end{thm}
	
	\begin{proof}
Suppose $\iprd{\varphih^{m-1}}{1}=0$. It is clear that a necessary condition for solvability of \eqref{scheme-a-mean-zero} -- \eqref{scheme-b-mean-zero} is that 
	\begin{equation}
\left(\varphih^m,1\right) = \bigl(\varphih^{m-1},1\bigr) = 0,
	\end{equation}
as can be found by taking $\nu\equiv 1$ in \eqref{scheme-a-mean-zero}.   Now, let $ \varphih^m,\mu_{h,\star}^m  \in\Soh\times\Soh$ be a solution  of \eqref{nonlinear-1} -- \eqref{nonlinear-3}.  (The other variables may be regarded as auxiliary.)  Set
	\begin{equation}
\overline{\mu_h^m}:=\frac{1}{\varepsilon|\Omega|}\iprd{(\varphih^m+\avephio)^3 - \left(\varphih^m+\avephio\right)}{1} =\frac{1}{\varepsilon|\Omega|}\iprd{(\varphih^m+\avephio)^3 }{1} -\frac{\avephio}{\varepsilon} ,
	\label{chem-pot-average}
	\end{equation}
and define $\muh^m:=\mu_{h,\star}^m+\overline{\mu_h^m}$.  There is a one-to-one correspondence of the respective solution sets: $\varphih^m,\mu_{h,\star}^m \in\Soh\times\Soh$ is a solution to \eqref{nonlinear-1} -- \eqref{nonlinear-3}, if and only if $\varphih^m,\muh^m  \in\Soh\times S_h$ is a solution to \eqref{scheme-a-mean-zero} -- \eqref{scheme-e-mean-zero}, if and only if $ \phih^m,\muh^m \in S_h\times S_h$ is a solution to \eqref{scheme-a} -- \eqref{scheme-e}, where
	\begin{equation}
\phih^m = \varphih^m+\avephio,\quad \muh^m = \mu_{h,\star}^m+\overline{\mu_h^m}.
	\end{equation}
But \eqref{nonlinear-1} -- \eqref{nonlinear-3} admits a unique solution, which proves that \eqref{scheme-a} -- \eqref{scheme-e} and \eqref{scheme-a-mean-zero} -- \eqref{scheme-e-mean-zero} are uniquely solvable.
	\end{proof}

	\subsection{Unconditional Energy Stability}
	\label{subsec-energy-stability}

We now show that the solutions to our scheme enjoy stability properties that are similar to those of the PDE solutions, and, moreover, these properties hold regardless of the sizes of $h$ and $\tau$.  The first property, the unconditional energy stability, is a direct result of the convex decomposition represented in the scheme.
	
	\begin{lem}
	\label{lem-energy-law}
Let $(\phih^m, \muh^m, {\bf u}_h^m) \in S_h\times S_h\times {\bf X}_h$ be the unique solution of  \eqref{scheme-a}--\eqref{scheme-e}, with the other variables regarded as auxiliary.  Then the following energy law holds for any $h,\,  \tau >0$:
	\begin{align}
E\left({\bf u}_h^{\ell},\phih^\ell\right) &+\tau \varepsilon \sum_{m=1}^\ell \norm{\nabla\muh^m}{L^2}^2 + \tau \frac{\lambda}{\gamma} \sum_{m=1}^\ell \norm{\nabla\buh^m}{L^2}^2 + \tau \frac{\eta}{\gamma} \sum_{m=1}^\ell \norm{\buh^m}{L^2}^2 
	\nonumber
	\\
&+ \tau^2 \sum_{m=1}^\ell \Biggl\{ \, \frac{\varepsilon}{2} \norm{\nabla\left(\dtau  \phih^m\right)}{L^2}^2+ \frac{1}{2\gamma} \norm{\dtau\buh^m}{L^2}^2 + \frac{1}{4\varepsilon}\norm{\dtau( \phih^m)^2}{L^2}^2
	\nonumber
	\\
&\quad  + \frac{1}{2\varepsilon}\norm{ \phih^m \dtau \phih^m}{L^2}^2  +\frac{1}{2\varepsilon}\norm{\dtau \phih^m}{L^2}^2 +\frac{\theta}{2}\norm{\dtau\phih^m}{-1,h}^2\, \Biggr\} = E\left({\bf u}_h^0,\phih^0\right),
	\label{ConvSplitEnLaw}
	\end{align}
for all $0\leq \ell \leq M$.
	\end{lem}

	\begin{proof}
We first set $\nu= \muh^m$ in \eqref{scheme-a}, $\psi = \dtau \phih^m$ in \eqref{scheme-b}, $\zeta = -\mathsf{T}_h\left(\dtau \phih^m\right)$ in \eqref{scheme-c}, ${\bf v} = \frac{1}{\gamma}\buh^m$ in \eqref{scheme-d}, $q = \frac{1}{\gamma}p_h^m$ in \eqref{scheme-e}, to obtain
	\begin{eqnarray}
\iprd{\dtau \phih^m}{\muh^m} + \varepsilon \norm{\nabla\muh^m}{L^2}^2 + \bform{\phih^{m-1}}{\buh^m}{\muh^m} &=  & 0,
	\label{tested-energy-1}
	\\
	\frac{1}{\varepsilon}\iprd{\left(\phih^m\right)^3 -\phih^{m-1}}{\dtau\phih^m} + \varepsilon \,\aiprd{\phih^m}{\dtau\phih^m} -\iprd{\muh^m}{\dtau\phih^m} + \iprd{\xi_h^m}{\dtau\phih^m} &=& 0,
	\label{tested-energy-2}
	\\
-\aiprd{\xi_h^m}{\mathsf{T}_h\left(\dtau \phih^m\right)} + \theta \,\iprd{\phih^m-\avephio}{\mathsf{T}_h\left(\dtau \phih^m\right)} & =  & 0 ,
	\label{tested-energy-3} 
	\\
\frac{1}{\gamma}\iprd{\delta_\tau{\bf u}^m_h}{\buh^m}+ \frac{\lambda}{\gamma} \norm{\nabla\buh^m}{L^2}^2 + \frac{\eta}{\gamma} \norm{\buh^m}{L^2}^2 -\frac{1}{\gamma}\cform{\buh^m}{p^m_h} -\bform{\phih^{m-1}}{\buh^m}{\muh^m} & = & 0  ,
	\label{tested-energy-4} 
	\\
\frac{1}{\gamma}\cform{\buh^m}{p^m} & = & 0 .
	\label{tested-energy-5}
	\end{eqnarray}
Combining \eqref{tested-energy-1} -- \eqref{tested-energy-5}, using the identities 
	\begin{align*}
\left(\dtau \buh^m, \buh^m \right) = &\frac12\,\left[\,\dtau\norm{\buh^m}{L^2}^2 + \tau \norm{\dtau \buh^m}{L^2}^2 \,\right],
\\
\aiprd{\phih^m}{\dtau\phih^m}  = &\frac12\,\left[\,\dtau\norm{\nabla\phih^m}{L^2}^2  +\tau \norm{\nabla\dtau \phih^m}{L^2}^2 \,\right],
	\\
\iprd{\left(\phih^m\right)^3-\phih^{m-1}}{\dtau \phih^m} = & \frac14\, \dtau \norm{\left( \phih^m\right)^2-1}{L^2}^2  +\frac{\tau}4 \Bigl[\norm{\dtau( \phih^m)^2}{L^2}^2
	\\
& +2\norm{\phih^m \dtau \phih^m}{L^2}^2 +2\norm{\dtau \phih^m}{L^2}^2 \, \Bigr],
	\\
\iprdmone{\phih^m-\avephio}{\dtau\phih^m} =& \frac{1}{2}\left[\dtau\norm{\phih^m-\avephio}{-1,h}^2 + \tau\norm{\dtau\phih^m}{-1,h}^2\right], 
	\end{align*}
and applying the operator $\tau\sum_{m=1}^\ell$ to the combined equation, the result is obtained.
	\end{proof}

The discrete energy law immediately implies the following uniform (in  $h$ and $\tau$) \emph{a priori} estimates for $\phih^m$, $\muh^m$, and $\buh^m$. Note that, from this point, we will not track the dependence of the estimates on the interface parameter $\varepsilon>0$, though this may be of importance, especially if $\varepsilon$ is made smaller.

	\begin{lem}
	\label{lem-a-priori-stability-trivial}
Let $(\phih^m, \muh^m, {\bf u}_h^m) \in S_h\times S_h\times {\bf X}_h$ be the unique solution of  \eqref{scheme-a}--\eqref{scheme-e}.  Suppose that $E\left(\buh^0,\phih^0\right)<C$, independent of $h$.  Then the following estimates hold for any $h,\, \tau>0$:
	\begin{align}
\max_{0\leq m\leq M} \left[ \norm{\buh^m}{L^2}^2+ \norm{\nabla\phih^m}{L^2}^2 + \norm{\left( \phih^m\right)^2-1}{L^2}^2 + \norm{\phih^m-\avephio}{-1,h}^2\right] &\leq C, 
	\label{Linf-u-phi-discrete}
	\\
\max_{0\leq m\leq M}\left[\norm{\phih^m}{L^4}^4 +\norm{\phih^m}{L^2}^2  +\norm{\phih^m}{H^1}^2 \right] &\le C,
	\label{LinfH1phi-discrete}
	\\
\tau \sum_{m=1}^M\bigg[ \norm{\nabla\muh^m}{L^2}^2 +  \norm{\nabla\buh^m}{L^2}^2 + \norm{\buh^m}{L^2}^2 \bigg] &\leq C , 
	\label{sum-mu-u-discrete}
	\\
\sum_{m=1}^M \bigg[ \norm{\nabla\left(\phih^m-\phih^{m-1}\right)}{L^2}^2 + \norm{ \phih^m-\phih^{m-1}}{L^2}^2 +  \norm{ \phih^m( \phih^m-\phih^{m-1})}{L^2}^2
	\nonumber
	\\
+ \norm{( \phih^m)^2-(\phih^{m-1})^2}{L^2}^2 + \norm{\phih^m-\phih^{m-1}}{-1,h}^2 + \norm{\buh^m - \buh^{m-1}}{L^2}^2 \bigg] &\leq C , 
	\label{sum-phi-u-discrete}
	\end{align}
for some constant $C>0$ that is independent of $h$, $\tau$, and $T$.
	\end{lem}
	
We are able to prove the next set of \emph{a priori} stability estimates without any restrictions of $h$ and $\tau$.  Before we begin, we will need the discrete  Laplacian, $\Delta_h: S_h \to \Soh$, which is defined as follows:  for any $v_h\in S_h$, $\Delta_h v_h\in \Soh$ denotes the unique solution to the problem
	\begin{equation}
\iprd{\Delta_h v_h}{\chi} = -\aiprd{v_h}{\chi},  \quad\forall \, \,\chi\in S_h.
	\label{discrete-laplacian}
	\end{equation}
In particular, setting $\chi = \Delta_h v_h$ in \eqref{discrete-laplacian}, we obtain  
	\begin{displaymath}
\norm{\Delta_h v_h}{L^2}^2 = -\aiprd{v_h}{ \Delta_hv_h} .
	\end{displaymath}

	\begin{lem}
	\label{lem-improved-a-priori-stabilities}
Let $(\phih^m, \muh^m, {\bf u}_h^m) \in S_h\times S_h\times {\bf X}_h$ be the unique solution of  \eqref{scheme-a}--\eqref{scheme-e}, with the other variables regarded as auxiliary.  Suppose that $E\left(\buh^0,\phih^0\right) < C$ independent of $h$ and that $T\ge 1$ (for simplicity). The following estimates hold for any $h,\, \tau >0$:
	\begin{equation}
\tau \sum_{m=1}^M \bigg[ \norm{\dtau \phih^m}{H^{-1}}^2 +\norm{\dtau \phih^m}{-1,h}^2+\norm{ \Delta_h \phih^m}{L^2}^2 + \norm{\muh^m}{L^2}^2 + \norm{\phih^m}{L^\infty}^{\frac{4(6-d)}{d}} \bigg] \le TC,
	\label{sum-3D-good}
	\end{equation}
for some constant $C>0$ that is independent of $h$, $\tau$, and $T$.
	\end{lem}

	\begin{proof}
Let $\mathcal{Q}_h: L^2(\Omega) \to S_h$ be the $L^2$ projection, \emph{i.e.}, $\iprd{\mathcal{Q}_h \nu-\nu}{\chi}=0$, for all $\chi\in S_h$. Suppose $\nu\in\mathring{H}^1(\Omega)$. Then, using \eqref{Linf-u-phi-discrete} and Sobolev embeddings,
	\begin{align}
\iprd{\dtau\phih^m}{\nu}&= \iprd{\dtau\phih^m}{\mathcal{Q}_h\nu} 
	\\
&=-\varepsilon \ \bigl(\nabla \mu_h^m,\nabla \mathcal{Q}_h\nu \bigr) 
- \bform{\phih^{m-1}}{\buh^m}{\mathcal{Q}_h\nu} 
	\\
&\leq \varepsilon\norm{\nabla\mu_h^m}{L^2} \norm{\nabla \mathcal{Q}_h\nu}{L^2} + \norm{\nabla\phih^{m-1}}{L^2}\norm{\buh^m}{L^4} \norm{ \mathcal{Q}_h\nu}{L^4}
	\\
&\leq C\left[\varepsilon\norm{\nabla\mu_h^m}{L^2} + \norm{\buh^m}{H^1} \right]\, \norm{\nabla \mathcal{Q}_h \nu}{L^2}
	\\
&\leq C\left[\varepsilon\norm{\nabla\muh^m}{L^2} + \norm{\buh^m}{H^1} \right]\, \norm{\nabla \nu}{L^2},
	\end{align}	
where we used the $H^1$ stability of the $L^2$ projection in the last step. Applying $\tau\sum_{m=1}^M$ gives (\ref{sum-3D-good}.1) -- which, in our notation, is the bound on the first term of the left side of \eqref{sum-3D-good}. The estimate (\ref{sum-3D-good}.2) follows from the inequality $\norm{\nu}{-1,h} \le \norm{\nu}{H^{-1}}$, which holds for all $\nu\in L^2(\Omega)$.
	
Setting $\psi_h= \Delta_h \phih^m$ in \eqref{scheme-b} and using the  definition of $\Delta_h\phih^m$, we get
	\begin{align*}
\varepsilon \norm{ \Delta_h \phih^m}{L^2}^2  &= -  \varepsilon \,\aiprd{\phih^m}{ \Delta_h \phih^m}
	\\
&=-\iprd{\muh^m}{\Delta_h \phih^m} + \varepsilon^{-1}\iprd{\left(\phih^m\right)^3-\phih^{m-1}}{\Delta_h \phih^m} +\iprd{\xi_h^m}{\Delta_h \phih^m} 
	\\
& \le\aiprd{ \muh^m}{\phih^m } + \varepsilon^{-1}\left(\frac{ \varepsilon^2}{2} \norm{ \Delta_h \phih^m}{L^2}^2  + \frac{1}{2 \varepsilon^2} \norm{\left(\phih^m\right)^3-\phih^{m-1}}{L^2}^2 \right)-\aiprd{\xi_h^m}{\phih^m}
	\\
&\leq \frac12\norm{\nabla\muh^m}{L^2}^2 +\frac12 \norm{\nabla\phih^m}{L^2}^2  +\frac{ \varepsilon}2 \norm{ \Delta_h \phih^m}{L^2}^2 +C\norm{\left(\phih^m\right)^3-\phih^{m-1}}{L^2}^2 - \theta \,\iprd{\phih^m-\avephio}{\phih^m}
	\\
&\leq \frac12\norm{\nabla\muh^m}{L^2}^2 + C \norm{\nabla \phih^m}{L^2}^2  +\frac{ \varepsilon}2 \norm{ \Delta_h \phih^m}{L^2}^2 +C\norm{\left(\phih^m\right)^3-\phih^{m-1}}{L^2}^2 + C \norm{\phih^m - \avephio}{-1,h}^2.
	\end{align*}
Hence,
	\begin{equation}
\varepsilon \norm{ \Delta_h \phih^m}{L^2}^2 \leq \norm{\nabla \muh^m}{L^2}^2  + C \norm{\nabla \phih^m}{L^2}^2 + C\norm{\left(\phih^m\right)^3-\phih^{m-1}}{L^2}^2 + C \norm{\phih^m - \avephio}{-1,h}^2.
	\label{discrete-laplacian-phi-middle}
	\end{equation}
Now using \eqref{LinfH1phi-discrete}, we have
	\begin{align} 
\norm{(\phih^m)^3-\phih^{m-1}}{L^2}^2 &\le 2\left(\norm{\phih^m}{L^6}^6 + \norm{\phih^{m-1}}{L^2}^2\right) 
	\nonumber
	\\
& \le C \norm{\phih^m}{H^1}^6 +C  
	\nonumber
	\\
& \le C,
	\label{nonlinear-control}
	\end{align}
where we used the embedding $H^1(\Omega) \hookrightarrow L^6(\Omega)$, for $d = 2, 3$. Putting the last two inequalities together, we have 
	\begin{equation}
\varepsilon \norm{ \Delta_h \phih^m}{L^2}^2 \leq \norm{\nabla \muh^m}{L^2}^2  +C.
	\end{equation}
Applying $\tau\sum_{m=1}^M$, estimate (\ref{sum-3D-good}.3) now follows from (\ref{sum-mu-u-discrete}.1).

Now, take $\psi = \muh^m$ in \eqref{scheme-b}.  Then, using \eqref{Linf-u-phi-discrete} and \eqref{nonlinear-control}, we have
	\begin{eqnarray}
\norm{\muh^m}{L^2}^2 &\le& \varepsilon^{-1}	\norm{\left(\phih^m\right)^3 - \phih^{m-1}}{L^2} \norm{\muh^m}{L^2} +\varepsilon\norm{\nabla\phih^m}{L^2}\norm{\nabla\muh^m}{L^2} +  \norm{\xi_h^m}{L^2}\norm{\mu_h^m}{L^2}
	\nonumber
	\\
&\le& \frac{1}{\varepsilon^2}\norm{\left(\phih^m\right)^3 - \phih^{m-1}}{L^2}^2 +\frac{1}{4}\norm{\muh^m}{L^2}^2 +\frac{\varepsilon}{2}\norm{\nabla\phih^m}{L^2}^2 + \frac{\varepsilon}{2}\norm{\nabla\muh^m}{L^2}^2 
	\nonumber
	\\
&&+ C\norm{\nabla\xi_h^m}{L^2}^2 + \frac{1}{4} \norm{\muh^m}{L^2}
	\nonumber
	\\
&\le&  C +\frac{1}{2}\norm{\muh^m}{L^2}^2  + \frac{\varepsilon}{2} \norm{\nabla\muh^m}{L^2}^2 + C\norm{\nabla\xi_h^m}{L^2}^2 
	\nonumber
	\\
&\le&  C +\frac{1}{2}\norm{\muh^m}{L^2}^2  + \frac{\varepsilon}{2} \norm{\nabla\muh^m}{L^2}^2 + C \norm{\phih^m-\avephio}{-1,h}^2
	\nonumber
	\\
&\le&  C +\frac{1}{2}\norm{\muh^m}{L^2}^2  + \frac{\varepsilon}{2} \norm{\nabla\muh^m}{L^2}^2 .
	\nonumber
	\end{eqnarray}
Hence
	\begin{equation}
\norm{\muh^m}{L^2}^2 \le C + \varepsilon\norm{\nabla\muh^m}{L^2}^2.
	\end{equation}
Applying $\tau\sum_{m=1}^M$, estimate (\ref{sum-3D-good}.4) now follows from (\ref{sum-mu-u-discrete}.1).

To prove estimate (\ref{sum-3D-good}.5), we use the discrete Gagliardo-Nirenberg inequality:
	\begin{equation}
\|\phih^m\|_{L^\infty} \leq C\|\Delta_h \phih^m\|_{L^2}^{\frac{d}{2(6-d)}}
\,\|\phih^m\|_{L^6}^{\frac{3(4-d)}{2(6-d)}} + C\|\phih^m\|_{L^6} \qquad (d=2,3) .
	\label{infinity-bound}
	\end{equation}
Applying $\tau\sum_{m=1}^M$ and using $H^1(\Omega) \hookrightarrow L^6(\Omega)$, (\ref{LinfH1phi-discrete}.3) and  (\ref{sum-3D-good}.3), estimate (\ref{sum-3D-good}.5) follows.  
	\end{proof}
	
	\begin{lem}
	\label{lem-a-priori-stability}
Let $(\phih^m, \muh^m, {\bf u}_h^m) \in S_h\times S_h\times {\bf X}_h$ be the unique solution of  \eqref{scheme-a}--\eqref{scheme-e}, with the other variables regarded as auxiliary.  Suppose that $E\left(\buh^0,\phih^0\right), \norm{\muh^0}{L^2}^2 < C$ independent of $h$,  where $\muh^0$ is defined below in \eqref{initial-chem-pot}, $d=2,3$, and that $T\ge 1$ (for simplicity). The following estimates hold for any $h,\, \tau >0$:
	\begin{align}
\tau \sum_{m=1}^M \norm{\dtau\phih^m}{L^2}^2 & \le CT ,
	\label{sum-dtau-phi}
	\\
\max_{1\le m\le M} \bigg[ \norm{\muh^m}{L^2}^2 + \norm{\Delta_h \phih^m}{L^2}^2  + \norm{\phih^m}{L^\infty}^{\frac{4(6-d)}{d}} \bigg] &\le CT,
	\label{Linf-mu-phi} 
	\end{align}
for some constant $C>0$ that is independent of $h$, $\tau$, and $T$.
	\end{lem}
	
	\begin{proof}
	We prove (\ref{sum-dtau-phi}) and (\ref{Linf-mu-phi}.1) together.  To do so, we first define $\muh^0$ via
	\begin{equation}
	\label{initial-chem-pot}
\iprd{\muh^0}{\psi} := \varepsilon \,\aiprd{\phih^0}{\psi} + \varepsilon^{-1}\iprd{\left(\phih^0\right)^3 -\phih^0 }{\psi}  + \theta \,\iprd{\mathsf{T}_h\left(\phih^0-\avephio\right)}{\psi},
	\end{equation}
for all $\psi \in S_h$, and
	\begin{equation}
\dtau \phih^0 :\equiv 0\in S_h.
	\end{equation}	
Now, we subtract \eqref{scheme-b} from itself at consecutive time steps to obtain
	\begin{eqnarray}
\tau\iprd{\dtau\muh^m}{\psi} &=& \tau\varepsilon \,\aiprd{\dtau\phih^m}{\psi} + \varepsilon^{-1}\iprd{\left(\phih^m\right)^3-\left(\phih^{m-1}\right)^3}{\psi} 
	\nonumber
	\\
&& -\tau\varepsilon^{-1}\iprd{\dtau \phih^{m-1}}{\psi}  + \theta\tau \,\iprd{\mathsf{T}_h\left(\dtau\phih^m\right)}{\psi} ,	
	\label{scheme-b-staggered}
	\end{eqnarray}
for all $\psi \in S_h$, which is well-defined for all $1\le m \le M$.  Taking $\psi = \muh^m$ in \eqref{scheme-b-staggered} and $\nu = -\tau\dtau\phih^m$ in \eqref{scheme-a} and adding the results yields 
	\begin{eqnarray}
\tau \,\iprd{\dtau\muh^m}{\muh^m} +\tau \norm{\dtau\phih^m}{L^2}^2  &=& \tau \varepsilon^{-1}\iprd{\dtau\phih^m\left\{ \left(\phih^m\right)^2 +\phih^m\phih^{m-1} + \left(\phih^{m-1}\right)^2\right\}}{\muh^m}  
	\nonumber
	\\
&&- \tau \varepsilon^{-1} \iprd{\dtau \phih^{m-1}}{\muh^m} + \theta\tau \,\iprd{\mathsf{T}_h\left(\dtau\phih^m\right)}{\muh^m - \overline{\mu_h^m}}
	\nonumber
	\\
&&- \tau \,\bform{\phih^{m-1}}{\buh^m}{\dtau\phih^m}
	\nonumber
	\\
&\le& \tau\varepsilon^{-1} \norm{\left(\phih^m\right)^2 + \phih^m\phih^{m-1} + \left(\phih^{m-1}\right)^2 }{L^3}\norm{\muh^m}{L^6} \norm{\dtau\phih^m}{L^2}
	\nonumber
	\\
&&+ \tau \varepsilon^{-1} \norm{\nabla\muh^m}{L^2} \norm{\dtau\phih^{m-1}}{-1,h}
	\nonumber
	\\
&& + \theta\tau \norm{\nabla{\mathsf{T}_h\left(\dtau\phih^m\right)}}{L^2} \norm{\muh^m - \overline{\mu_h^m}}{-1,h}
	\nonumber
	\\
&&- \tau \,\bform{\phih^{m-1}}{\buh^m}{\dtau\phih^m}
	\nonumber
	\\
&\le&  C \tau \norm{\left(\phih^m\right)^2 +\phih^m\phih^{m-1} + \left(\phih^{m-1}\right)^2 }{L^3}^2\norm{\muh^m}{H^1}^2  +\frac{\tau}{4} \norm{\dtau\phih^m}{L^2}^2
	\nonumber
	\\
& & + C\tau \norm{\nabla\muh^m}{L^2}^2 +  C\tau \norm{\dtau\phih^{m-1}}{-1,h}^2
	\nonumber
	\\
&& + C\tau  \norm{\nabla{\mathsf{T}_h\left(\dtau\phih^m\right)}}{L^2}^2 + C\tau\norm{\muh^m - \overline{\mu_h^m}}{-1,h}^2
	\nonumber
	\\
&&- \tau \,\bform{\phih^{m-1}}{\buh^m}{\dtau\phih^m}
	\nonumber
	\\
&\le&  C \tau \left( \norm{\phih^m}{L^6}^4 +\norm{\phih^{m-1}}{L^6}^4\right)\norm{\muh^m}{H^1}^2  +\frac{\tau}{4} \norm{\dtau\phih^m}{L^2}^2
	\nonumber
	\\
& & + C\tau \norm{\nabla\muh^m}{L^2}^2 +  C\tau \norm{\dtau\phih^{m-1}}{-1,h}^2
	\nonumber
	\\
&& + C\tau  \norm{\dtau\phih^m}{-1,h}^2 + C\tau\norm{\nabla\muh^m}{L^2}^2 - \tau \,\bform{\phih^{m-1}}{\buh^m}{\dtau\phih^m}
	\nonumber
	\\
&\le&   C \tau \norm{\muh^m}{H^1}^2  +\frac{\tau}{4} \norm{\dtau\phih^m}{L^2}^2  +  C\tau \norm{\dtau\phih^{m-1}}{-1,h}^2
	\nonumber
	\\
& &   + C\tau  \norm{\dtau\phih^m}{-1,h}^2  - \tau \,\bform{\phih^{m-1}}{\buh^m}{\dtau\phih^m}
	\label{estimate-2.56}
	\end{eqnarray}
where we have used $H^1(\Omega) \hookrightarrow L^6(\Omega)$, Lemma~\ref{lem-negative-norm-discrete}, and \eqref{LinfH1phi-discrete}.  

Now we bound the trilinear form $b(\, \cdot \, , \, \cdot \, , \, \cdot \, )$. To do so, we note the discrete estimate
	\begin{equation}
\norm{\nabla \nu_h}{L^4}  \leq C\left( \norm{\nabla \nu_h}{L^2}+\norm{\Delta_h \nu_h}{L^2}\right)^{\frac{d}{4}} \norm{\nabla \nu_h}{L^2}^{\frac{4-d}{4}} \quad \forall \, \nu_h \in S_h,\quad d= 2,3,
	\label{grad-v-L4-bound}
	\end{equation}
and Ladyzhenskaya inequality
	\begin{equation}
\norm{\bu}{L^4}\le C\norm{\bu}{L^2}^{\frac{4-d}{4}}\norm{\nabla\bu}{L^2}^{\frac{d}{4}}\qquad \forall \,  \, \bu \in {\bf H}_0^1(\Omega) ,\quad d= 2,3.
	\label{Lady}
	\end{equation} 
Using Holder's inequality, \eqref{grad-v-L4-bound},  \eqref{Lady}, (\ref{Linf-u-phi-discrete}.1), and (\ref{Linf-u-phi-discrete}.2)

	\begin{eqnarray}
\left| \bform{\phih^{m-1}}{\buh^m}{\dtau\phih^m}\right| &\le & \norm{\nabla\phih^{m-1}}{L^4}\norm{{\bf u}_h^m}{L^4}\norm{\dtau\phih^m}{L^2}
	\nonumber
	\\
&\le & C \norm{\dtau\phih^m}{L^2} \norm{\nabla{\bf u}_h^m}{L^2}\left(\norm{\nabla\phih^{m-1}}{L^2}+\norm{\Delta_h\phih^{m-1}}{L^2}\right)
	\nonumber
	\\
&\le & \frac{1}{4} \norm{\dtau\phih^m}{L^2}^2 + C \norm{\nabla{\bf u}_h^m}{L^2}^2 + C \norm{\nabla{\bf u}_h^m}{L^2}^2 \norm{\Delta_h\phih^{m-1}}{L^2}^2 .
	\label{b-form-stability-estimate-a}
 	\end{eqnarray}
Setting $\psi_h= \Delta_h \phih^m$ in \eqref{scheme-b} and \eqref{initial-chem-pot} and using the  definition of $\Delta_h\phih^m$, it follows that
	\begin{equation}
\norm{ \Delta_h \phih^m}{L^2}^2 \leq C\norm{\muh^m}{L^2}^2  +C , \quad 0 \le m \le M, 
	\label{lap-phi-bounded-by-mu}
	\end{equation}
so that, for $1\le m\le M$,
	\begin{equation}
\left| \bform{\phih^{m-1}}{\buh^m}{\dtau\phih^m}\right| \le  \frac{1}{4} \norm{\dtau\phih^m}{L^2}^2 +  C \norm{\nabla{\bf u}_h^m}{L^2}^2 + C \norm{\nabla{\bf u}_h^m}{L^2}^2 \norm{\muh^{m-1}}{L^2}^2 .
	\end{equation}
 
Thus,
	\begin{eqnarray}
\tau\iprd{\dtau\muh^m}{\muh^m} +\frac{\tau}{2} \norm{\dtau\phih^m}{L^2}^2  &\le&   C\tau \norm{\muh^m}{H^1}^2 +  C\tau \norm{\dtau\phih^{m-1}}{-1,h}^2 + C\tau  \norm{\dtau\phih^m}{-1,h}^2  
	\nonumber
	\\
&& + C\tau\norm{\nabla{\bf u}_h^m}{L^2}^2 \norm{\muh^{m-1}}{L^2}^2+ C\tau \norm{\nabla{\bf u}_h^m}{L^2}^2.
	\end{eqnarray}
Applying $\sum_{m=1}^\ell$, and using \eqref{sum-mu-u-discrete}, (\ref{sum-3D-good}), $\dtau\phih^0 \equiv 0$, and the identity
	\begin{equation}
\tau\iprd{\dtau\muh^m}{\muh^m} = \iprd{\muh^m-\muh^{m-1}}{\muh^m} = \frac{1}{2}\norm{\muh^m}{L^2}^2 +\frac{1}{2}\norm{\muh^m-\muh^{m-1}}{L^2}^2-\frac{1}{2}\norm{\muh^{m-1}}{L^2}^2 ,
	\end{equation}
we conclude
	\begin{equation}
\frac{1}{2}\norm{\muh^\ell}{L^2}^2 - \frac{1}{2}\norm{\muh^0}{L^2}^2 +\frac{\tau}{2}\sum_{m=1}^\ell\norm{\dtau\phih^m}{L^2}^2 \le CT + C\tau\sum_{m=0}^{\ell-1} \norm{\nabla{\bf u}_h^{m+1}}{L^2}^2 \norm{\muh^m}{L^2}^2 .
	\end{equation}	
Since the estimate is explicit with respect to $\left\{\norm{\muh^m}{L^2}^2\right\}$ and $\tau \sum_{m=1}^M\norm{\nabla{\bf u}_h^m}{L^2}^2 \le C$, we may appeal directly to the discrete Gronwall inequality in Lemma~\ref{lem-discrete-gronwall}.  Estimates \eqref{sum-dtau-phi} and (\ref{Linf-mu-phi}.1) follow immediately. Estimate (\ref{Linf-mu-phi}.2) follows from (\ref{Linf-mu-phi}.1) and \eqref{lap-phi-bounded-by-mu}.  Estimate (\ref{Linf-mu-phi}.3) follows from \eqref{infinity-bound}, the embedding $H^1(\Omega) \hookrightarrow L^6(\Omega)$, (\ref{LinfH1phi-discrete}.3), and (\ref{Linf-mu-phi}.2). 
	\end{proof}

	\begin{rem}
The idea for controlling the time-lagged $\norm{\Delta_h\phih^{m-1}}{L^2}^2$ term in \eqref{b-form-stability-estimate-a} using the discrete Gronwall inequality was inspired by a similar technique from a recent paper by G.~Gr\"{u}n~\cite{grun13}, which deals with a different PDE system (as well as a different numerical method) from that examined here and is not concerned with error estimates.
	\end{rem}

	\section{Error Estimates for the Fully Discrete Convex Splitting Scheme}
	\label{sec-error-estimates}
	
For the error estimates that we pursue in this section, we shall assume that weak solutions have the additional regularities
	\begin{align}
\phi &\in H^2\bigl(0,T; L^2(\Omega)\bigr) \cap L^\infty\left(0,T;W^{1,6}(\Omega)\right)\cap H^1\bigl(0,T;H^{q+1}(\Omega)\bigr),
	\nonumber
	\\
\xi &\in L^2 \bigl(0,T; H^{q+1}(\Omega) \bigr),
	\nonumber
	\\
\mu &\in L^\infty\bigl(0,T;H^1(\Omega)\bigr)\cap L^2\bigl(0,T;H^{q+1}(\Omega)\bigr),
	\label{regularities}
	\\
{\bf u} &\in H^2\left(0,T;{\bf L}^2(\Omega)\right) \cap L^{\infty}\bigl(0,T; {\bf L}^4(\Omega)\bigr)\cap H^1 \bigl(0,T;{\bf H}^{q+1}(\Omega)\bigr),
	\nonumber
	\\
p &\in L^2 \bigl(0,T;H^q(\Omega)\cap L_0^2(\Omega)\bigr),
	\nonumber	
	\end{align}
where $q\ge 1$. Of course, some of these regularities are redundant because of embeddings.  For the Darcy-Stokes projection we have
	\begin{equation}
\norm{{\bf P}_h {\bf u}-{\bf u}}{H^1} + \norm{P_h p - p}{L^2} \le C h^q\left(\left|{\bf u}\right|_{H^{q+1}} + \left|p\right|_{H^q}\right) . 
	\end{equation}
	
Weak solutions $(\phi,\mu)$ to $\eqref{weak-mch-a} - \eqref{weak-mch-e}$ with the higher regularities \eqref{regularities} solve the following variational problem:
	\begin{subequations}
	\begin{align}
\iprd{\partial_t \phi}{\nu} + \varepsilon \,\aiprd{\mu}{\nu} + \bform{\phi}{{\bf u}}{\nu} &= 0  &&\qquad \forall \, \nu \in H^1(\Omega),
	\label{weak-mch-error-a} 
	\\
\iprd{\mu}{\psi}-\varepsilon \,\aiprd{\phi}{\psi} - \varepsilon^{-1}\iprd{\phi^3-\phi}{\psi}  - \iprd{\xi}{\psi}
&= 0  &&\qquad \forall \, \psi\in H^1(\Omega),
	\label{weak-mch-error-b} 
	\\
\aiprd{\xi}{\zeta}-\theta \,\iprd{\phi-\avephio}{\zeta}&= 0  &&\qquad \forall \, \zeta\in H^1(\Omega),
	\label{weak-mch-error-c}
	\\
\iprd{\partial_t{\bf u}}{{\bf v}}+ \lambda\,\aiprd{{\bf u}}{{\bf v}} + \eta \,\iprd{ {\bf u}}{ {\bf v}} -\cform{{\bf v}}{p} - \gamma \,\bform{\phi}{{\bf v}}{\mu}&= 0  &&\qquad \forall \, {\bf v}\in \bf{H}_0^1(\Omega),
	\label{weak-mch-error-d} 
	\\
\cform{{\bf u}}{q} &= 0  &&\qquad \forall \, q\in L_0^2(\Omega).
	\label{weak-mch-error-e}
	\end{align}
	\end{subequations}

Define $L_{\tau} \phi(t) := \phi(t-\tau)$, $\dtau \phi(t) :=\frac{\phi(t)-L_{\tau}\phi(t)}{\tau}$, the backward difference operator, and 
	\begin{eqnarray}
	\label{eq:five}
\eAphi := \phi-R_h \phi , \quad \eAmu := \mu -R_h \mu, \quad \eAu := {\bf u} - {\bf P}_h {\bf u} , \quad  \eAxi := \xi - R_h \xi ,
	\\
\sigphi := \dtau R_h \phi - \partial_t\phi, \quad \sigu := \dtau {\bf P}_h {\bf u} - \partial_t {\bf u}.
	\end{eqnarray}
Then, for all $\nu,\psi, \zeta \in S_h, {\bf v} \in {\bf X}_h,$ and $q \in \Soh$,
	\begin{subequations}
	\begin{align}
\iprd{\dtau R_h\phi}{\nu} + \varepsilon \, \aiprd{R_h\mu}{\nu} = & \iprd{\sigphi}{\nu} - \bform{\phi}{{\bf u}}{\nu},
	\label{weak-mch-error3-a} 
	\\
\varepsilon \, \aiprd{R_h\phi}{\psi}  - \iprd{R_h\mu}{\psi} + \iprd{R_h\xi}{\psi} = & \iprd{\eAmu}{\psi}  - \varepsilon^{-1}\iprd{\phi^3-L_{\tau}\phi}{\psi}
	\nonumber
	\\
& + \frac{\tau}{\varepsilon}\iprd{\dtau \phi}{\psi} - \iprd{\eAxi}{\psi},
	\label{weak-mch-error3-b} 
	\\
\aiprd{R_h \xi}{\zeta}-\theta \,\iprd{R_h \phi-\avephio}{\zeta} = & \, \theta \, \iprd{\eAphi}{\zeta},
	\label{weak-mch-error3-c}
	\\
\iprd{\dtau {\bf P}_h{\bf u}}{{\bf v}} + \lambda \,\aiprd{{\bf P}_h{\bf u}}{{\bf v}} + \eta \,\iprd{{\bf P}_h{\bf u}}{ {\bf v}} -\cform{{\bf v}}{P_h p} = & \iprd{\sigu}{{\bf v}} + \gamma \,\bform{\phi}{{\bf v}}{\mu},
	\label{weak-mch-error3-d} 
	\\
\cform{{\bf P}_h{\bf u}}{q} = & 0 .
	\label{weak-mch-error3-e}
	\end{align}
	\end{subequations}

Define the piecewise constant (in time) functions, for $m=1,\dots M$ and for $t\in(t_{m-1},t_m]$, 
	\begin{equation*}
\hatphi(t):=\phih^m, \quad \hatmu(t):=\muh^m, \quad \hatu(t):=\buh^m , \quad \hatp(t):=p_h^m ,  \quad \hatxi(t):=\xih^m,
	\end{equation*}
where $\phih^m$, $\muh^m$, $\xih^m$, $\buh^m$, and $p_h^m$ are the solutions of the fully discrete convex-splitting scheme \eqref{scheme-a} -- \eqref{scheme-e}. Thus, for $0\le t\le T$, and all $\nu,\psi, \zeta \in S_h, {\bf v} \in {\bf X}_h,$ and $q \in \Soh$, 
	\begin{subequations}
	\begin{align}
\iprd{\dtau \hatphi}{\nu} + \varepsilon \,\aiprd{\hatmu}{\nu} = \,& - \bform{L_{\tau}\hatphi}{\hatu}{\nu},
	\label{scheme-error-a}
	\\
\varepsilon \,\aiprd{\hatphi}{\psi}  - \iprd{\hatmu}{\psi} + \iprd{\hatxi}{\psi} 
= \,& - \varepsilon^{-1}\iprd{\hatphi^3- L_{\tau}\hatphi}{\psi},
	\label{scheme-error-b}
	\\
\aiprd{\hatxi}{\zeta}-\theta\iprd{\hatphi-\avephio}{\zeta}  = \,& 0 ,
	\label{scheme-error-c} 
	\\
\iprd{\delta_\tau \hatu}{{\bf v}}+ \lambda \,\aiprd{\hatu}{{\bf v}} + \eta \,\iprd{ \hatu}{ {\bf v}} -\cform{{\bf v}}{\hatp}  = \,& \gamma \,\bform{L_{\tau}\hatphi}{{\bf v}}{\hatmu},
	\label{scheme-error-d} 
	\\
\cform{\hatu}{q} = \,& 0 .
	\label{scheme-error-e} 
	\end{align}
	\end{subequations}
Now, let us define 
	\begin{eqnarray}
\ehphi := R_h \phi - \hatphi, \ \ephi := \phi - \hatphi, \ \ehmu := R_h \mu - \hatmu , \ \ehxi := R_h \xi - \hatxi ,
	\\
\ehu := {\bf P}_h {\bf u} - \hatu, \ \ehp := P_h p - \hatp .
	\end{eqnarray}
Subtracting \eqref{scheme-error-a} - \eqref{scheme-error-e} from \eqref{weak-mch-error3-a} - \eqref{weak-mch-error3-e}, we have, for all $\nu,\psi, \zeta \in S_h, {\bf v} \in {\bf X}_h,$ and $q \in \Soh$,
	\begin{subequations}
	\begin{align}
\iprd{\dtau \ehphi}{\nu} + \varepsilon \,\aiprd{\ehmu}{\nu} = & \iprd{\sigphi}{\nu} - \bform{\phi}{{\bf u}}{\nu}  + \bform{L_{\tau} \hatphi}{\hatu}{\nu},
	\label{error-a} 
	\\
\varepsilon \, \aiprd{\ehphi}{\psi}  - \iprd{\ehmu}{\psi} + \iprd{\ehxi}{\psi}
= & \iprd{\eAmu}{\psi} + \frac{\tau}{\varepsilon} \iprd{ \dtau \phi}{\psi} + \varepsilon^{-1}\iprd{ L_{\tau} \ephi}{\psi} 
	\nonumber
	\\
&- \iprd{\eAxi}{\psi} - \varepsilon^{-1}\iprd{\phi^3 - \hatphi^3}{\psi},
	\label{error-b} 
	\\
\aiprd{\ehxi}{\zeta}-\theta \,\iprd{\ehphi}{\zeta} = & \,\theta \, \iprd{\eAphi}{\zeta},
	\label{error-c}
	\\
\iprd{\dtau \ehu}{{\bf v}} + \lambda \,\aiprd{\ehu}{{\bf v}} + \eta \,\iprd{\ehu}{ {\bf v}} -\cform{{\bf v}}{\ehp} = & \iprd{\sigu}{{\bf v}} + \gamma \,\bform{\phi}{{\bf v}}{\mu}  - \gamma \,\bform{L_{\tau} \hatphi}{{\bf v}}{\hatmu},
	\label{error-d} 
	\\
\cform{\ehu}{q} &= 0  .
	\label{error-e}
	\end{align}
	\end{subequations}
	
Setting $\nu = \ehmu$ in (\ref{error-a}), $\psi = \dtau \ehphi$ in (\ref{error-b}), $\zeta = - \mathsf{T}_h\left(\dtau \ehphi \right)$ in (\ref{error-c}), ${\bf v} = \frac{1}{\gamma} \ehu$ in (\ref{error-d}), and $q =  \frac{1}{\gamma} \ehp$ in (\ref{error-e}) and adding the resulting equations we obtain
	\begin{align}
\varepsilon \,\aiprd{\ehphi}{\dtau\ehphi} &+ \theta \, \iprdmone{\ehphi}{\dtau \ehphi} 
+ \frac{1}{\gamma} \iprd{\dtau \ehu}{\ehu} +\varepsilon \norm{\nabla \ehmu}{L^2}^2 +  \frac{\lambda}{\gamma} \norm{\nabla \ehu}{L^2}^2 + \frac{\eta}{\gamma} \norm{\ehu}{L^2}^2
\nonumber
\\
& = \iprd{\sigphi}{\ehmu} + \frac{1}{\gamma} \iprd{\sigu}{\ehu} - \varepsilon^{-1}\iprd{\phi^3-\hatphi^3 - L_\tau\ephi}{\dtau\ehphi}  
	\nonumber 
	\\
& + \iprd{\eAmu}{\dtau \ehphi} + \frac{\tau}{\varepsilon}\iprd{\dtau \phi}{\dtau \ehphi} - \iprd{\eAxi}{\dtau \ehphi} - \theta \, \iprd{\eAphi}{\mathsf{T}_h \left(\dtau \ehphi \right)} 
	\nonumber
	\\
& - \bform{\phi}{{\bf u}}{\ehmu} + \bform{L_{\tau} \hatphi}{\hatu}{\ehmu} + \bform{\phi}{\ehu}{\mu} - \bform{L_{\tau} \hatphi}{\ehu}{\hatmu}.
	\label{error-eq-3}
	\end{align}
The last expression is the key error equation.  We now proceed to estimate the terms on the right hand side of \eqref{error-eq-3}. 

	\begin{lem}
	\label{truncation-errors}
Suppose that $(\phi, \mu,{\bf u})$ is a weak solution to \eqref{weak-mch-error-a} -- \eqref{weak-mch-error-e}, with the additional regularities \eqref{regularities}.  Then, for any $h$, $\tau >0$, there exists $C>0$, independent of $h$ and $\tau$, such that
	\begin{align}
&\norm{\sigphi(t)}{L^2}^2 \le  C\frac{h^{2q+2}}{\tau} \int_{t-\tau}^{t} \norm{\partial_s\phi(s)}{H^{q+1}}^2  ds  +  \frac{\tau}{3} \int_{t-\tau}^{t}\norm{\partial_{ss}\phi(s)}{L^2}^2 ds,
	\nonumber
	\\
&\norm{\sigu(t)}{L^2}^2 \le C\frac{h^{2q+2}}{\tau} \int_{t-\tau}^{t} \norm{\partial_s{\bf u}(s)}{H^{q+1}}^2  ds  +  \frac{\tau}{3} \int_{t-\tau}^{t}\norm{\partial_{ss}{\bf u}(s)}{L^2}^2 ds ,
	\end{align}
for all $t\in(\tau,T]$.
	\end{lem}
	
	\begin{proof}
We write $\sigphi = \sigphi_1+ \sigphi_2$, $\sigphi_1 := \dtau R_h\phi - \dtau\phi$,  $\sigphi_2 := \dtau\phi - \partial_t \phi$.  Then
	\begin{align}
\norm{\sigphi_1(t)}{L^2}^2 &=  \norm{\frac{1}{\tau}\int_{t-\tau}^{t}  \partial_s \left(R_h \phi(s)- \phi(s) \right) ds}{L^2}^2
	\nonumber
	\\
&= \frac{1}{\tau^2} \norm{\int_{t-\tau}^{t}   \left(R_h \partial_s\phi(s)- \partial_s\phi(s) \right) ds}{L^2}^2
	\nonumber
	\\
& \le \frac{1}{\tau^2} \int_\Omega\int_{t-\tau}^{t} 1^2 ds  \int_{t - \tau}^{t} \left(R_h \partial_s \phi(s)- \partial_s \phi(s) \right)^2 ds \, d{\bf x}
	\nonumber
	\\
& = \frac{1}{\tau} \int_{t-\tau}^{t} \norm{ R_h \partial_s \phi(s)- \partial_s \phi(s)}{L^2}^2 ds 
	\nonumber
	\\
& \le  C\frac{h^{2q+2}}{\tau} \int_{t-\tau}^{t} \norm{\partial_s\phi(s)}{H^{q+1}}^2 \, ds.
	\end{align}
By Taylor's theorem
	\begin{align}
\norm{\sigphi_2(t)}{L^2}^2 &=  \norm{\frac{1}{\tau}\int_{t-\tau}^{t}\partial_{ss}\phi(s)(t-s)ds}{L^2}^2
	\nonumber
	\\
&\le \frac{1}{\tau^2}\int_\Omega \left[\int_{t-\tau}^{t}(t-s)^2\, ds \int_{t-\tau}^{t}\left(\partial_{ss}\phi(s)\right)^2\, ds \right]   d{\bf x}
	\nonumber
	\\
&=  \frac{1}{\tau^2}\int_{t-\tau}^{t}(t-s)^2\, ds \int_{t-\tau}^{t}\norm{\partial_{ss}\phi(s)}{L^2}^2 ds
	\nonumber
	\\
&=  \frac{\tau^3}{3\tau^2} \int_{t-\tau}^{t}\norm{\partial_{ss}\phi(s)}{L^2}^2 ds.
	\end{align}
Using the triangle inequality, the result for $\norm{\sigphi(t)}{L^2}^2$ follows. A similar proof can be constructed for $\norm{\sigu(t)}{L^2}^2$.
	\end{proof}
	
We need a technical lemma that will be used a number of times.
	\begin{lem}
	\label{lem:technical}
Suppose $g \in H^1(\Omega)$, and $v \in \Soh$.  Then
	\begin{equation}
\left|\iprd{g}{v}\right| \le C \norm{\nabla g}{L^2} \, \norm{v}{-1,h}  ,
	\end{equation}
for some $C>0$ that is independent of $h$.
	\end{lem}
	\begin{proof}
If $g\in S_h$, we can apply Lemma~\ref{lem-negative-norm-discrete} directly.  Otherwise, using the triangle inequality,  the Cauchy-Schwarz inequality, and Lemma~\ref{lem-negative-norm-discrete},
	\begin{equation}
\left|\iprd{g}{v}\right| \le \left|\iprd{g - R_h g}{v}\right| +\left|\iprd{ R_h g}{v}\right| \le \norm{g-R_h g}{L^2} \norm{v}{L^2} + \norm{\nabla R_h g}{L^2} \norm{v}{-1,h}.
	\end{equation}
Using the standard elliptic projection estimate
	\begin{equation}
\norm{g-R_h g}{L^2} \le C h \norm{\nabla g}{L^2},
	\end{equation}
we have
	\begin{equation}
\left|\iprd{g}{v}\right| \le  C h \norm{\nabla g}{L^2}  \norm{v}{L^2} + \norm{\nabla R_h g}{L^2} \norm{v}{-1,h}.
	\end{equation}
Finally, using the (uniform) inverse estimate $h\norm{v}{L^2} \le C \norm{v}{-1,h}$ from Lemma~\ref{lem-negative-norm-discrete}, and the stability of the elliptic projection, $\norm{\nabla R_h g}{L^2} \le C \norm{\nabla g}{L^2}$, we have the result.
	\end{proof}

	\begin{lem}
	\label{nonlinear-estimate}
Suppose that $(\phi, \mu,{\bf u})$ is a weak solution to \eqref{weak-mch-error-a} -- \eqref{weak-mch-error-e}, with the additional regularities \eqref{regularities}. Then, for any $h$, $\tau >0$,
	\begin{equation}
\norm{\nabla \left(\phi^3 - \hat{\phi}^3\right)}{L^2} \le C \norm{\nabla \ephi}{L^2} ,
	\end{equation}
where $\ephi := \phi - \hat{\phi}$.
	\end{lem}
	\begin{proof}
For  $t\in[0,T]$, 
	\begin{align}
\norm{\nabla\left(\phi^3-\hat{\phi}^3 \right)}{L^2} &\le 3\norm{\hat{\phi}^2\nabla\ephi}{L^2} + 3\norm{\nabla\phi \left(\phi+\hat{\phi}\right)\ephi}{L^2}
	\nonumber
	\\
&\le 3\norm{\hat{\phi}}{L^\infty}^2\norm{\nabla\ephi}{L^2}+  3\norm{\nabla\phi}{L^6}\norm{\phi+\hat{\phi}}{L^6}\norm{\ephi}{L^6}
	\nonumber
	\\
&\le 3\left(\norm{\hat{\phi}}{L^\infty}^2 +  C\norm{\nabla\phi}{L^6}\norm{\phi+\hat{\phi}}{H^1}\right)\norm{\nabla\ephi}{L^2} \le C \norm{\nabla\ephi}{L^2},
	\end{align}
where $C>0$ is independent of $t\in[0,T]$. Then, using the unconditional \emph{a priori} estimates in Lemmas~\ref{lem-improved-a-priori-stabilities} and \ref{lem-a-priori-stability} and the assumption that $\phi\in L^\infty\left(0,T;W^{1,6}(\Omega)\right)$, the result follows.
	\end{proof}
	
	\begin{lem}
	\label{lem-error-1}
Suppose that $(\phi, \mu,{\bf u})$ is a weak solution to \eqref{weak-mch-error-a} -- \eqref{weak-mch-error-e}, with the additional regularities \eqref{regularities}.  Then, for any $h$, $\tau >0$, and any $\alpha > 0$ there exists a constant $C=C(\alpha)>0$, independent of $h$ and $\tau$, such that 
	\begin{align}
\frac{\varepsilon}{2} \norm{\nabla \ehmu}{L^2}^2 + \varepsilon \,\aiprd{\ehphi}{\dtau\ehphi} &+ \theta \, \iprdmone{\ehphi}{\dtau \ehphi} + \iprd{\dtau \ehu}{\ehu} + \frac{\lambda}{2 \gamma} \norm{\nabla \ehu}{L^2}^2 + \frac{\eta}{2 \gamma} \norm{\ehu}{L^2}^2 
	\nonumber
	\\
&\le C \norm{\nabla \ehphi}{L^2}^2 + C \norm{\nabla L_\tau \ehphi}{L^2}^2 + \alpha \norm{\dtau \ehphi}{-1,h}^2  + C \mathcal{R},
	\label{alpha-estimate}
	\end{align}
for any $t\in(\tau ,T]$, where $\mathcal{R}$ is the consistency term
	\begin{align}
\mathcal{R}(t) =& \frac{h^{2q+2}}{\tau} \int_{t-\tau}^{t} \norm{\partial_s\phi(s)}{H^{q+1}}^2  ds  + \frac{\tau}{3} \int_{t-\tau}^{t}\norm{\partial_{ss}\phi(s)}{L^2}^2 ds + \tau \int_{t - \tau}^{t} \norm{\nabla\partial_s\phi(s)}{L^2}^2 ds
	\nonumber
	\\
&+h^{2q}\left| \mu \right|_{H^{q+1}}^2+h^{2q}\left| \phi \right|_{H^{q+1}}^2 +h^{2q+2}\left| \phi \right|_{H^{q+1}}^2 + h^{2q}\left| L_\tau\phi \right|_{H^{q+1}}^2  + h^{2q}\left| \xi \right|_{H^{q+1}}^2
	\nonumber
	\\
& + \frac{h^{2q+2}}{\tau} \int_{t-\tau}^{t} \norm{\partial_s{\bf u}(s)}{H^{q+1}}^2  ds + \frac{\tau}{3} \int_{t-\tau}^{t}\norm{\partial_{ss}{\bf u}(s)}{L^2}^2 ds + h^{2q}\left(\left|{\bf u}\right|_{H^{q+1}}^2 + \left|p\right|_{H^q}^2\right).
	\label{consistency}
	\end{align}
	\end{lem}
	
	\begin{proof}
Using Lemmas~\ref{truncation-errors} and \ref{lem:technical}, the Cauchy-Schwarz inequality, the definition above, and the fact that $\iprd{\sigphi}{1} = 0$, we get the following estimates:
	\begin{align}
\left|\iprd{\sigphi}{\ehmu}\right| &\le \norm{\sigphi}{-1,h} \norm{\nabla\ehmu}{L^2}
	\nonumber
	\\
&\le C \norm{\sigphi}{L^2} \norm{\nabla \ehmu}{L^2}
	\nonumber
	\\
&\le C\norm{\sigphi}{L^2}^2 +\frac{\varepsilon}{10} \norm{\nabla \ehmu}{L^2}^2
	\nonumber
	\\
&\le C \left( \frac{h^{2q+2}}{\tau} \int_{t-\tau}^{t} \norm{\partial_s\phi(s)}{H^{q+1}}^2  ds + \frac{\tau}{3} \int_{t-\tau}^{t}\norm{\partial_{ss}\phi(s)}{L^2}^2 ds \right) + \frac{\varepsilon}{10} \norm{\nabla \ehmu}{L^2}^2
	\label{error-estimate-1}
	\end{align}
and, similarly, 
	\begin{equation}
\left|\iprd{\sigu}{\ehu}\right| \le C \left( \frac{h^{2q+2}}{\tau} \int_{t-\tau}^{t} \norm{\partial_s{\bf u}(s)}{H^{q+1}}^2  ds + \frac{\tau}{3} \int_{t-\tau}^{t}\norm{\partial_{ss}{\bf u}(s)}{L^2}^2 ds \right)+ \frac{\eta}{2 \gamma} \norm{\ehu}{L^2}^2.
	\label{error-estimate-2}
	\end{equation}
	
Now, from the standard finite element approximation theory
	\begin{equation*}
\norm{\nabla \eAmu}{L^2} = \norm{\nabla (R_h \mu - \mu)}{L^2} \leq C h^q\left| \mu \right|_{H^{q+1}} .
	\end{equation*}
Applying Lemma~\ref{lem:technical} and the last estimate
	\begin{eqnarray}
\left|\iprd{\eAmu}{\dtau \ehphi}\right| &\le& C \norm{\nabla\eAmu}{L^2} \, \norm{\dtau \ehphi}{-1,h} 
	\nonumber
	\\
&\le& C h^q\left| \mu \right|_{H^{q+1}} \norm{\dtau \ehphi}{-1,h}
	\nonumber
	\\
&\le& Ch^{2q}\left| \mu \right|_{H^{q+1}}^2 + \frac{\alpha}{6} \norm{\dtau \ehphi}{-1,h}^2 
	\label{error-estimate-3}
	\end{eqnarray}
and, similarly,
	\begin{equation}
\left|\iprd{\eAxi}{\dtau \ehphi}\right| \le Ch^{2q}\left| \xi \right|_{H^{q+1}}^2 + \frac{\alpha}{6} \norm{\dtau \ehphi}{-1,h}^2.
	\end{equation}
	
Now, it follows that
	\begin{equation}
\norm{ \tau \nabla \dtau \phi(t)}{L^2}^2 \le  \tau \int_{t-\tau}^{t} \norm{\nabla\partial_s\phi(s)}{L^2}^2 ds
	\label{lag-estimate}
	\end{equation}
and, therefore,
	\begin{eqnarray}
\frac{\tau}{\varepsilon}\left|\iprd{\dtau \phi}{\dtau\ehphi}\right| &\le& \frac{1}{\varepsilon} \norm{\tau\nabla \dtau \phi}{L^2} \, \norm{\dtau \ehphi}{-1,h}
	\nonumber
	\\
&\le& C\tau \int_{t - \tau}^{t} \norm{\nabla\partial_s\phi(s)}{L^2}^2 ds  + \frac{\alpha}{6} \norm{\dtau \ehphi}{-1,h}^2 .
	\label{error-estimate-4}
	\end{eqnarray}

Using Lemmas~\ref{lem:technical} and \ref{nonlinear-estimate} , as well as $\ephi =\eAphi+\ehphi$ and a standard error estimate,
	\begin{eqnarray}
\frac{1}{\varepsilon}\left|\iprd{\phi^3-\hat{\phi}^3}{\dtau \ehphi}\right| &\le& C \norm{\nabla \left(\phi^3-\hatphi^3 \right)}{L^2} \, \norm{\dtau \ehphi}{-1,h}
	\nonumber
	\\
&\le& C \norm{\nabla \left(\phi^3-\hatphi^3 \right)}{L^2}^2 +  \frac{\alpha}{6}  \norm{\dtau \ehphi}{-1,h}^2
	\nonumber
	\\
&\le& C \norm{\nabla \ephi}{L^2}^2 +  \frac{\alpha}{6}  \norm{\dtau \ehphi}{-1,h}^2
	\nonumber
	\\
&\le& C \norm{\nabla \eAphi}{L^2}^2 +C \norm{\nabla \ehphi}{L^2}^2 + \frac{\alpha}{6}  \norm{\dtau \ehphi}{-1,h}^2
	\nonumber
	\\
&\le& C h^{2q}\left| \phi \right|_{H^{q+1}}^2 + C \norm{\nabla \ehphi}{L^2}^2 + \frac{\alpha}{6}  \norm{\dtau \ehphi}{-1,h}^2 .
	\label{error-estimate-5}
	\end{eqnarray}
	
With similar steps as in the last estimate,
	\begin{eqnarray}
\frac{1}{\varepsilon}\left|\iprd{L_\tau\ephi }{\dtau \ehphi}\right| &\le& C \norm{\nabla L_\tau \ephi}{L^2} \, \norm{\dtau \ehphi}{-1,h}
	\nonumber
	\\
&\le& C h^{2q}\left| L_\tau\phi \right|_{H^{q+1}}^2 +C \norm{\nabla L_\tau \ehphi}{L^2}^2+ \frac{\alpha}{6}\norm{\dtau \ehphi}{-1,h}^2  .
	\label{error-estimate-6}
	\end{eqnarray}
	
Using the estimate
	\begin{equation}
\norm{\mathsf{T}_h\left(\dtau \ehphi \right)}{L^2}^2 \le C \norm{\nabla \mathsf{T}_h\left(\dtau \ehphi \right)}{L^2}^2 = C \norm{\dtau \ehphi}{-1,h}^2 ,
\nonumber
	\end{equation}
we obtain
	\begin{align}
\left| \theta \, \iprd{\eAphi}{\mathsf{T}_h \left(\dtau \ehphi \right)} \right| &\le \theta \norm{\eAphi}{L^2} \,\norm{\mathsf{T}_h \left(\dtau \ehphi \right)}{L^2}
	\nonumber
	\\
& \le C h^{q+1}\left| \phi \right|_{H^{q+1}} \norm{\dtau \ehphi}{-1,h}
	\nonumber
	\\
& \le C h^{2q+2}\left| \phi \right|_{H^{q+1}}^2 + \frac{\alpha}{6} \norm{\dtau \ehphi}{-1,h}^2.
	\label{error-estimate-7}
	\end{align}
	
Now we consider the trilinear terms. Adding and subtracting the appropriate terms and using the triangle inequality gives 
	\begin{align}
\Bigg| - \bform{\phi}{{\bf u}}{\ehmu} &+ \bform{L_{\tau} \hatphi}{\hatu}{\ehmu} +  \bform{\phi}{\ehu}{\mu} - \bform{L_{\tau} \hatphi}{\ehu}{\hatmu}  \Bigg| 
	\nonumber
	\\
& \le \left|\bform{\eAphi}{{\bf u}}{\ehmu} \right| + \left| \bform{L_{\tau} \ehphi}{{\bf u}}{\ehmu} \right| + \left| \bform{\tau \dtau R_h \phi}{{\bf u}}{\ehmu} \right| + \left| \bform{L_{\tau} \hatphi}{\eAu}{\ehmu} \right|
	\nonumber
	\\
& + \left| \bform{\eAphi}{\ehu}{\mu} \right| + \left| \bform{L_{\tau} \ehphi}{\ehu}{\mu} \right| + \left| \bform{\tau\dtau R_h \phi}{\ehu}{\mu} \right| + \left| \bform{L_{\tau} \hatphi}{\ehu}{\eAmu} \right|.
	\end{align}
With the assumption ${\bf u} \in L^\infty\left(0,T; {\bf L}^4(\Omega)\right)$ we have
	\begin{eqnarray}
\left|\bform{\eAphi}{{\bf u}}{\ehmu}\right| &\le& \norm{\nabla \eAphi}{L^2} \norm{{\bf u}}{L^4} \norm{\ehmu}{L^4}
	\nonumber
	\\
&\le& C \norm{\nabla \eAphi}{L^2}^2 + \frac{\varepsilon}{10} \norm{\nabla \ehmu}{L^2}^2
	\nonumber
	\\
&\le& C h^{2q}\left| \phi \right|_{H^{q+1}}^2 + \frac{\varepsilon}{10} \norm{\nabla \ehmu}{L^2}^2 ,
	\label{error-estimate-8}
	\end{eqnarray}
as well as
	\begin{eqnarray}
\left|\bform{L_{\tau} \ehphi}{{\bf u}}{\ehmu}\right| &\le& \norm{\nabla L_{\tau} \ehphi}{L^2} \norm{{\bf u}}{L^4} \norm{\ehmu}{L^4}
	\nonumber
	\\
&\le& C \norm{\nabla L_{\tau} \ehphi}{L^2}^2 + \frac{\varepsilon}{10} \norm{\nabla \ehmu}{L^2}^2.
	\label{error-estimate-9}
	\end{eqnarray}
Using the stability of the elliptic projection, and reusing estimate \eqref{lag-estimate}, and ${\bf u} \in L^\infty\left(0,T; {\bf L}^4(\Omega)\right)$
	\begin{eqnarray}
\left| \bform{\tau\dtau R_h \phi}{{\bf u}}{\ehmu}\right| &\le&  \norm{\nabla \tau\dtau R_h \phi}{L^2} \norm{{\bf u}}{L^4} \norm{\ehmu}{L^4}
	\nonumber
	\\
&\le& C  \norm{\tau\nabla\dtau R_h\phi}{L^2} \norm{\nabla \ehmu}{L^2}
	\nonumber
	\\
&\le& C  \norm{\tau\nabla\dtau \phi}{L^2} \norm{\nabla \ehmu}{L^2}
	\nonumber
	\\
&\le& C  \norm{\tau\nabla \dtau \phi}{L^2}^2 + \frac{\varepsilon}{10} \norm{\nabla \ehmu}{L^2}^2
	\nonumber
	\\
&\le& C \tau \int_{t-\tau}^{t} \norm{\nabla\partial_s\phi(s)}{L^2}^2 ds + \frac{\varepsilon}{10} \norm{\nabla \ehmu}{L^2}^2.
	\label{error-estimate-10}
	\end{eqnarray}
Using (\ref{Linf-mu-phi}.3), 
	\begin{eqnarray}
\left|\bform{L_{\tau} \hatphi}{\eAu}{\ehmu} \right| &\le& \norm{\nabla L_{\tau} \hatphi}{L^2} \norm{\eAu}{L^4} \norm{\ehmu}{L^4}
	\nonumber
	\\
&\le& C \norm{\eAu}{H^1}^2 + \frac{\varepsilon}{10} \norm{\nabla \ehmu}{L^2}^2
	\nonumber
	\\
&\le& C h^{2q}\left(\left|{\bf u}\right|_{H^{q+1}}^2 + \left|p\right|_{H^q}^2\right) + \frac{\varepsilon}{10} \norm{\nabla \ehmu}{L^2}^2
	\label{error-estimate-11}.
	\end{eqnarray}
Since we assume $\mu\in L^\infty\left(0,T;H^1(\Omega) \right)$,
	\begin{eqnarray}
\left|\bform{\eAphi}{\ehu}{\mu} \right| &\le& \norm{\nabla \eAphi}{L^2} \norm{\ehu}{L^4} \norm{\mu}{L^4} 
	\nonumber
	\\
&\le& C \norm{\nabla \eAphi}{L^2} \norm{\nabla \ehu}{L^2} \norm{\mu}{H^1} 
	\nonumber
	\\
&\le& C \norm{\nabla \eAphi}{L^2}^2 + \frac{\lambda}{8 \gamma} \norm{\nabla \ehu}{L^2}^2
	\nonumber
	\\
&\le& C h^{2q}\left|\phi\right|_{H^{q+1}}^2 + \frac{\lambda}{8 \gamma} \norm{\nabla \ehu}{L^2}^2,
	\label{error-estimate-12}
	\end{eqnarray}
and
	\begin{eqnarray}
\left|\bform{L_{\tau} \ehphi}{\ehu}{\mu} \right| &\le&  \norm{\nabla L_{\tau} \ehphi}{L^2} \norm{\ehu}{L^4} \norm{\mu}{L^4}
	\nonumber
	\\
&\le& C \norm{\nabla L_{\tau} \ehphi}{L^2} \norm{\nabla \ehu}{L^2} \norm{\mu}{H^1}
	\nonumber
	\\
&\le& C \norm{\nabla L_{\tau} \ehphi}{L^2}^2 + \frac{\lambda}{8 \gamma} \norm{\nabla \ehu}{L^2}^2.
	\label{error-estimate-13}
	\end{eqnarray}
Again, using $\mu \in L^\infty\left(0,T; H^1(\Omega)\right)$, the stability of the elliptic projection, and reusing estimate \eqref{lag-estimate}
	\begin{eqnarray}
\left| \,\bform{\tau\dtau R_h \phi}{\ehu}{\mu}\right| &\le&  \norm{\nabla \tau\dtau R_h \phi}{L^2} \norm{\ehu}{L^4} \norm{\mu}{H^1} 
	\nonumber
	\\
&\le& C   \norm{\tau\nabla \dtau R_h\phi}{L^2} \norm{\nabla \ehu}{L^2} 
	\nonumber
	\\
&\le& C  \norm{\tau\nabla \dtau \phi}{L^2} \norm{\nabla \ehu}{L^2} 
\nonumber
\\
&\le& C \tau \int_{t-\tau}^{t} \norm{\nabla\partial_s\phi(s)}{L^2}^2 ds + \frac{\lambda}{8 \gamma} \norm{\nabla \ehu}{L^2}^2.
	\label{error-estimate-14}
	\end{eqnarray}
Finally,	
	\begin{eqnarray}
\left|\bform{L_{\tau} \hatphi}{\ehu}{\eAmu} \right| &\le&  \norm{\nabla L_{\tau} \hatphi}{L^2} \norm{\ehu}{L^4} \norm{\eAmu}{L^4}
	\nonumber
	\\
&\le& C \norm{\nabla \ehu}{L^2} \norm{\nabla \eAmu}{L^2}
	\nonumber
	\\
&\le& \frac{\lambda}{8 \gamma} \norm{\nabla \ehu}{L^2}^2 + C \norm{\nabla \eAmu}{L^2}^2
	\nonumber
	\\
&\le& \frac{\lambda}{8 \gamma} \norm{\nabla \ehu}{L^2}^2 + C h^{2q}\left|\mu\right|_{H^{q+1}}^2.
	\label{error-estimate-15}
	\end{eqnarray}

Combining the estimates \eqref{error-estimate-1} -- \eqref{error-estimate-15} with the error equation \eqref{error-eq-3} and using the triangle inequality, the result follows.
	\end{proof}
	
	\begin{lem}
	\label{lem-1,h-error-estimate}
Suppose that $(\phi, \mu,{\bf u})$ is a weak solution to \eqref{weak-mch-error-a} -- \eqref{weak-mch-error-e}, with the additional regularities \eqref{regularities}.  Then, for any $h$, $\tau >0$, there exists a constant $C>0$, independent of $h$ and $\tau$, such that
	\begin{equation}
\norm{\dtau \ehphi}{-1,h}^2  \le 7\varepsilon^2 \norm{\nabla \ehmu}{L^2}^2 + C \norm{\nabla L_{\tau} \ehphi}{L^2}^2+ 7C_2^2 \norm{\nabla \ehu}{L^2}^2 + C\mathcal{R},
 	\label{-1,h-error-estimate}
	\end{equation}
for any $t\in (\tau, T]$, where $C_2 = C_0^2 C_1$, $C_0$ is  the $H^1(\Omega) \hookrightarrow L^4(\Omega)$ Sobolev embedding constant, $C_1$ is a bound for $\max_{0\le t\le T}\norm{\nabla \hat\phi}{L^2}^2$, and $\mathcal{R}$ is the consistency term given in \eqref{consistency}.
	\end{lem}
	
	\begin{proof}
Setting $\nu = \mathsf{T}_h\left(\dtau \ehphi \right)$ in \eqref{error-a}, we have
	\begin{eqnarray}
\norm{\dtau \ehphi}{-1,h}^2  &=& - \varepsilon\, \aiprd{\ehmu}{\mathsf{T}_h\left(\dtau \ehphi \right)} + \iprd{\sigphi}{\mathsf{T}_h\left(\dtau \ehphi \right)} 
	\nonumber
	\\
&& - \,\bform{\phi}{{\bf u}}{\mathsf{T}_h\left(\dtau \ehphi \right)} + \bform{L_{\tau} \hatphi}{\hatu}{\mathsf{T}_h\left(\dtau \ehphi \right)} 
	\nonumber
	\\
&=& - \,\varepsilon \iprd{\ehmu}{\dtau \ehphi} + \iprd{\sigphi}{\mathsf{T}_h\left(\dtau \ehphi \right)} 
	\nonumber
	\\
&& - \,\bform{\eAphi}{{\bf u}}{\mathsf{T}_h\left(\dtau \ehphi \right)} - \bform{L_{\tau} \ehphi}{{\bf u}}{\mathsf{T}_h\left(\dtau \ehphi \right)}
	\nonumber
	\\
&& - \bform{\tau\dtau R_h \phi}{{\bf u}}{\mathsf{T}_h\left(\dtau \ehphi \right)} - \bform{L_{\tau} \hatphi}{\eAu}{\mathsf{T}_h\left(\dtau \ehphi \right)} - \bform{L_{\tau} \hatphi}{\ehu}{\mathsf{T}_h\left(\dtau \ehphi \right)}
	\nonumber
	\\
&\le& \varepsilon \norm{\nabla \ehmu}{L^2} \norm{\dtau \ehphi}{-1,h} + \norm{\sigphi}{L^2} \norm{\mathsf{T}_h\left(\dtau \ehphi \right)}{L^2}
	\nonumber
	\\
&& + \,\norm{\nabla \eAphi}{L^2} \norm{{\bf u}}{L^4} \norm{\mathsf{T}_h\left(\dtau \ehphi \right)}{L^4} + \norm{\nabla L_{\tau} \ehphi}{L^2} \norm{{\bf u}}{L^4} \norm{\mathsf{T}_h\left(\dtau \ehphi \right)}{L^4}
	\nonumber
	\\
&& + \norm{\tau \nabla \dtau R_h \phi}{L^2} \norm{{\bf u}}{L^4} \norm{\mathsf{T}_h\left(\dtau \ehphi \right)}{L^4} + \norm{\nabla L_{\tau} \hatphi}{L^2} \norm{\eAu}{L^4} \norm{\mathsf{T}_h\left(\dtau \ehphi \right)}{L^4}
	\nonumber
	\\
&& + \norm{\nabla L_{\tau} \hatphi}{L^2} \norm{\ehu}{L^4} \norm{\mathsf{T}_h\left(\dtau \ehphi \right)}{L^4}
	\nonumber
	\\
&\le& \frac{7\varepsilon^2}{2} \norm{\nabla \ehmu}{L^2}^2 + \frac{1}{14} \norm{\dtau \ehphi}{-1,h}^2 + C \norm{\sigphi}{L^2}^2 + \frac{1}{14} \norm{\dtau \ehphi}{-1,h}^2
	\nonumber
	\\
&& + \,C \norm{\nabla \eAphi}{L^2}^2 + \frac{1}{14} \norm{\dtau \ehphi}{-1,h}^2 + C \norm{\nabla L_{\tau} \ehphi}{L^2}^2 + \frac{1}{14} \norm{\dtau \ehphi}{-1,h}^2
\nonumber
\\
&& + \,C \tau^2 \,\norm{\nabla \dtau R_h \phi}{L^2}^2 + \frac{1}{14} \norm{\dtau \ehphi}{-1,h}^2 + C \norm{\eAu}{L^2}^2 + \frac{1}{14} \norm{\dtau \ehphi}{-1,h}^2
	\nonumber
	\\
&&+ \frac{7C_2^2}{2} \norm{\nabla \ehu}{L^2}^2 + \frac{1}{14} \norm{\dtau \ehphi}{-1,h}^2
\nonumber
\\
&\le& \frac12 \norm{\dtau \ehphi}{-1,h}^2 + \frac{7\varepsilon^2}{2} \norm{\nabla \ehmu}{L^2}^2 + \frac{7C_2^2}{2} \norm{\nabla \ehu}{L^2}^2 +C \norm{\nabla L_{\tau} \ehphi}{L^2}^2 + C\mathcal{R} ,
	\end{eqnarray}
where we have used Lemmas~\ref{lem-negative-norm-discrete} and \ref{truncation-errors}. The result now follows.
	\end{proof}

	\begin{lem}
	\label{lem-error-3}
Suppose that $(\phi, \mu,{\bf u})$ is a weak solution to \eqref{weak-mch-error-a} -- \eqref{weak-mch-error-e}, with the additional regularities \eqref{regularities}.  Then, for any $h$, $\tau >0$, there exists a constant $C>0$, independent of $h$ and $\tau$, such that 
	\begin{align}
\norm{\nabla \ehmu}{L^2}^2 +  \norm{\ehu}{H^1}^2 + \aiprd{\ehphi}{\dtau\ehphi} &+ \iprdmone{\ehphi}{\dtau \ehphi} + \iprd{\dtau \ehu}{\ehu} 
	\nonumber
	\\
&\le C \norm{\nabla \ehphi}{L^2}^2 + C \norm{\nabla L_\tau \ehphi}{L^2}^2 + C\mathcal{R} .
	\end{align}
	\end{lem}
	\begin{proof}
This follows upon combining  the last two lemmas and choosing $\alpha$ in \eqref{alpha-estimate} appropriately.
	\end{proof}

Using the last lemma, we are ready to show the main convergence result for our convex-splitting scheme.
	\begin{thm}
	\label{thm-error-estimate}
Suppose $(\phi, \mu,{\bf u})$ is a weak solution to \eqref{weak-mch-error-a} -- \eqref{weak-mch-error-e}, with the additional regularities \eqref{regularities}.  Then, provided $0<\tau <\tau_0$, for some $\tau_0$ sufficiently small,
	\begin{align}
\max_{1\le m \le M} \left[\norm{\nabla\ehphi(t_m)}{L^2}^2 + \norm{\ehphi(t_m)}{-1,h}^2 + \norm{\ehu(t_m)}{L^2}\right] & 
	\nonumber
	\\
+ \tau\sum_{m=1}^M\left[\norm{\nabla\ehmu(t_m)}{L^2}^2 + \norm{\ehu(t_m)}{H^1}^2 + \norm{\dtau\ehphi(t_m)}{-1,h}^2 \right] &\le C(T)(\tau^2+h^{2q}) 
	\end{align}
for some $C(T)>0$ that is independent of $\tau$ and $h$.
	\end{thm}

	\begin{proof}
Setting $t=t_m$ and using Lemma~\ref{lem-error-3} and the arithmetic-geometric mean inequality, we have
	\begin{align}
\dtau\norm{\nabla\ehphi(t_m)}{L^2}^2 + \dtau\norm{\ehphi(t_m)}{-1,h}^2 + \dtau\norm{\ehu(t_m)}{L^2}^2 &
	\nonumber
	\\
+ \norm{\nabla\ehmu(t_m)}{L^2}^2 +\norm{\ehu(t_m)}{H^1}^2  & \le  \,C\norm{\nabla\ehphi(t_m)}{L^2}^2 + C\norm{\nabla\ehphi(t_{m-1})}{L^2}^2 
	\nonumber
	\\
& \hspace{0.2in} + C\mathcal{R}(t_m).
	\end{align}
Let $1 < \ell \le M$. Applying $\tau\sum_{m=1}^\ell$ and using $\ehphi(t_0) \equiv 0$, $\ehu(t_0) \equiv {\bf 0}$, 
	\begin{align}
\norm{\nabla\ehphi(t_\ell)}{L^2}^2 + \norm{\ehphi(t_\ell)}{-1,h}^2 + \norm{\ehu(t_\ell)}{L^2}^2 &
	\nonumber
	\\
+  \tau \sum_{m=1}^\ell \left[\norm{\nabla\ehmu(t_m)}{L^2}^2 + \norm{\ehu(t_m)}{H^1}^2  \right] &\le  C\tau\sum_{m=1}^\ell \mathcal{R}(t_m) + C_1\tau\sum_{m=1}^\ell\norm{\nabla\ehphi(t_m)}{L^2}^2 .
	\label{estimate-before-gronwall}
	\end{align}
If $0< \tau \le \tau_0:= \frac{1}{2C_1} < \frac{1}{C_1}$, it follows from the last estimate that
	\begin{eqnarray}
\norm{\nabla\ehphi(t_\ell)}{L^2}^2  &\le&  C\tau\sum_{m=1}^\ell \mathcal{R}(t_m)  + \frac{C_1\tau}{1-C_1\tau} \sum_{m=1}^{\ell-1}\norm{\nabla\ehphi
(t_m)}{L^2}^2 
	\nonumber
	\\
&\le& C(\tau^2+h^{2q}) + C\tau \sum_{m=1}^{\ell-1}\norm{\nabla\ehphi(t_m)}{L^2}^2 ,
	\label{pre-gronwall}
	\end{eqnarray}
where we have used the regularity assumptions to conclude $\tau\sum_{m=1}^M \mathcal{R}(t_m)\le C(\tau^2+h^{2q})$. Appealing to the discrete Gronwall inequality \eqref{gronwall-conclusion}, it follows that, for any $1 < \ell \le M$,
	\begin{equation}
\norm{\nabla\ehphi(t_\ell)}{L^2}^2 \le C(T)(\tau^2+h^{2q}).
	\label{post-gronwall}
	\end{equation}
Considering estimates \eqref{estimate-before-gronwall} and \eqref{-1,h-error-estimate} we get the desired result.
	\end{proof}
	
	\begin{rem}
From here it is straightforward to establish an optimal error estimate of the form
	\begin{align}
\max_{1\le m \le M}\left[ \norm{\nabla\ephi(t_m)}{L^2}^2 +\norm{\eu(t_m)}{L^2}^2\right] + \tau\sum_{m=1}^M \left[\norm{\nabla\emu(t_m)}{L^2}^2 + \norm{\nabla\eu(t_m)}{L^2}^2\right]  \le C(T)(\tau^2+h^{2q}) 
	\end{align}
using $\ephi = \eAphi + \ehphi$, \emph{et cetera}, the triangle inequality, and the standard spatial approximations. We omit the details for the sake of brevity.
	\end{rem}

	\section{Numerical Experiments}
	\label{sec:numerincal-experiments}
	
In this section, we provide some numerical experiments to gauge the accuracy and reliability of the fully discrete finite element method developed in the previous sections. We use a square domain $\Omega = (0,1)^2\subset \mathbb{R}^2$ and take ${\mathcal T}_h$ to be a regular triangulation of $\Omega$ consisting of right isosceles triangles. To refine the mesh, we assume that ${\mathcal T}_{\ell}, \ {\ell} = 0, 1, ..., L$, is an hierarchy of nested triangulations of $\Omega$ where ${\mathcal T}_{\ell}$, is obtained by subdividing the triangles of ${\mathcal T}_{\ell -1}$ into four congruent sub-triangles. Note that $h_{\ell -1} = 2h_{\ell}, \ {\ell} = 1, ..., L$ and that $\{{\mathcal T}_{\ell}\}$ is a quasi-uniform family. For the flow problem, we use the inf-sup-stable Taylor-Hood element where the ${\mathcal P}_1$ finite element space is used for the pressure and the $\left[{\mathcal P}_2\right]^2$ finite element space is used for the velocity. (We use a family of meshes ${\mathcal T}_h$ such that no triangle in the mesh has more than one edge on the boundary, as is usually required for the stability of the Taylor-Hood element~\cite{brenner08}.)  We use the ${\mathcal P}_1$ finite element space for the phase field and chemical potential.  In short, we take $q=1$.

We solve the scheme \eqref{scheme-a} -- \eqref{scheme-e} with the following parameters: $\lambda =1$, $\eta =1$, $\theta = 0$, and $\epsilon = 6.25 \times 10^{-2}$. The initial data for the phase field are taken to be
	\begin{equation}
\phi_{h}^0 = \mathcal{I}_h\left\{ \frac{1}{2}\Big(1.0-\cos(4.0\pi x)\Big)\cdot \Big(1.0-\cos(2.0\pi y)\Big)-1.0\right\} ,
	\end{equation}
where $\mathcal{I}_h :   H^2\left(\Omega\right) \to S_h$ is the standard nodal interpolation operator. Recall that our analysis does not specifically cover the use of the operator $\mathcal{I}_h$ in the initialization step.  But, since the error introduced by its use is optimal, a slight modification of the analysis show that this will lead to optimal rates of convergence overall.  (See Remark~\ref{rem:initial-projection}.)   The initial data for the velocity are taken as ${\bf u}_h^0 = {\bf 0}$. Values of the remaining parameters are given in the caption of Table~\ref{tab1}. To solve the system of equations above numerically, we are using the finite element libraries from the FEniCS Project~\cite{fenics12}.  We solve the fully coupled system by a Picard-type iteration.  Namely, at a given time step we fix the velocity and pressure, then solve for $\phi_h$, $\mu_h$, and $\xi_h$.  With these updated, we then solve for the velocity and pressure.  This is repeated until convergence.

	\begin{table}[h!]
	\centering
	\begin{tabular}{ccccccccc}
$h_c$ & $h_f$ & $\norm{\delta_\phi}{H^1}$ & rate & $\norm{\delta_\mu}{H^1}$ & rate & $\norm{\delta_p}{H^1}$ & rate
	\\
	\hline
$\nicefrac{\sqrt{2}}{8}$ & $\nicefrac{\sqrt{2}}{16}$ & $1.988\times 10^0$ & -- & $2.705\times 10^0$ & -- & $3.732\times 10^0$ &  -- &
	\\
$\nicefrac{\sqrt{2}}{16}$ & $\nicefrac{\sqrt{2}}{32}$ & $9.149\times 10^{-1}$ & 1.09 & $1.309\times 10^0$ & 1.03 & $9.73\times 10^{-1}$ & 1.92
	\\
$\nicefrac{\sqrt{2}}{32}$ & $\nicefrac{\sqrt{2}}{64}$ & $4.483\times 10^{-1}$ & 1.02 & $6.216\times 10^{-1}$ & 1.05 & $9.417\times 10^{-1}$ & 1.02
	\\
$\nicefrac{\sqrt{2}}{64}$ & $\nicefrac{\sqrt{2}}{128}$ & $2.231\times 10^{-1}$ & 1.00 & $3.056\times 10^{-1}$ & 1.02 & $4.688\times 10^{-1}$ & 1.00
	\\
	\hline
	\end{tabular}
\caption{$H^1$ Cauchy convergence test. The final time is $T = 4.0\times 10^{-1}$, and the refinement path is taken to be 
$\tau = .001\sqrt{2}h$. The other parameters are $\varepsilon =6.25\times 10^{-2}$; $\Omega = (0,1)^2$.  The Cauchy difference is defined via $\delta_\phi := \phi_{h_f}-\phi_{h_c}$, where the approximations are evaluated at time $t=T$, and analogously for $\delta_\mu$, and $\delta_p$. (See the discussion in the text.) Since $q=1$, \emph{i.e.}, we use ${\mathcal P}_1$ elements for these variables, the norm of the Cauchy difference at $T$ is expected to be $\mathcal{O}(\tau_f)+\mathcal{O}\left(h_f\right) = \mathcal{O}\left(h_f\right)$.}
	\label{tab1}
	\end{table}
	
Note that source terms are not naturally present in the system of equations \eqref{eq:CH-mixed-a-alt} -- \eqref{eq:CH-mixed-f-alt}. Therefore, it is somewhat artificial to add them to the equations in attempt to manufacture exact solutions.  To get around the fact that we do not have possession of exact solutions, we measure error by a different means.  Specifically, we compute the rate at which the Cauchy difference, $\delta_\zeta := \zeta^{M_f}_{h_f} - \zeta^{M_c}_{h_c}$, converges to zero, where $h_f=2h_c$, $\tau_f = 2\tau_c$, and $\tau_fM_f = \tau_cM_c=T$. Then, using a linear refinement path, \emph{i.e.}, $\tau = Ch$, and assuming $q = 1$, we have
	\begin{equation}
\norm{\delta_\zeta}{H^1} = \norm{\zeta^{M_f}_{h_f} - \zeta^{M_c}_{h_c}}{H^1} \le \norm{\zeta^{M_f}_{h_f}-\zeta(T)}{H^1}+ \norm{\zeta^{M_c}_{h_c}-\zeta(T)}{H^1} = \mathcal{O}(h_f^q+\tau_f) = \mathcal{O}(h_f).
	\end{equation}
The results of the $H^1$ Cauchy error analysis are found in Table~\ref{tab1} and confirm first-order convergence in this case. Additionally, we have proved that (at the theoretical level) the energy is non-increasing at each time step.  This is observed in our computations, but, for the sake of brevity, we will suppress an extensive discussion of numerical energy dissipation.

	\appendix
	
	\section{Discrete Gronwall Inequality}
	
We will need the following discrete Gronwall inequality cited in ~\cite{heywood90,layton08}:
	\begin{lem}	
	\label{lem-discrete-gronwall}
Fix $T>0$, and suppose $\left\{a^m\right\}_{m=1}^M$, $\left\{b^m\right\}_{m=1}^M$ and $\left\{c^m\right\}_{m=1}^{M-1}$ are non-negative sequences such that $\tau\sum_{m=1}^{M-1} c^m \le C_1$, where $C_1$ is independent of $\tau$ and $M$, and $M\tau = T$.  Suppose that, for all $\tau>0$, 
	\begin{equation}
a^M + \tau \sum_{m=1}^{M} b^m \le C_2 +\tau \sum_{m=1}^{M-1} a^m c^m ,
	\label{gronwall-assumption}
	\end{equation}
where $C_2>0$ is a constant independent of $\tau$ and $M$.  Then, for all $\tau>0$, 
	\begin{equation}
a^M +\tau \sum_{m=1}^{M} b^m \le C_2 \exp \left(\tau\sum_{m=1}^{M-1}c^m \right) \le C_2\exp(C_1).
	\label{gronwall-conclusion}
	\end{equation}
	\end{lem}
Note that the sum on the right-hand-side of \eqref{gronwall-assumption} must be explicit.

	\section*{Acknowledgments}
SMW acknowledges partial financial support from NSF-DMS 1115390.

	\bibliographystyle{plain}
	\bibliography{MDSCHMixed}

	\end{document}